\documentclass[12pt]{amsart}

\usepackage{bold-extra}
\usepackage{amsfonts}
\usepackage{amsthm}
\usepackage{xfrac}
\usepackage{faktor}
\usepackage{amssymb}
\usepackage{enumitem}
\usepackage{microtype}

\title{Projectively coresolved Gorenstein flat dimension of groups}
\author{Dimitra-Dionysia Stergiopoulou}

\newtheorem{Lemma}{Lemma}[section]
\newtheorem{Proposition}[Lemma]{Proposition}
\newtheorem{Theorem}[Lemma]{Theorem}
\newtheorem{Corollary}[Lemma]{Corollary}
\newtheorem{Remark}[Lemma]{Remark}
\newtheorem{Definition}[Lemma]{Definition}

\begin{document}

\begin{abstract} 
 In this paper, we introduce and study the projectively coresolved Gorenstein flat dimension of a group $G$ over a commutative ring $R$ and we prove that this dimension enjoys all the properties of the cohomological and the Gorenstein cohomological dimension. We also provide good estimations for the Gorenstein global dimension of $RG$ in terms of this dimension and the Gorenstein global dimension of $R$. Moreover, we study special cases of groups, such as $\textsc{\textbf{lh}}\mathfrak{F}$-groups, and show that for such a group every Gorenstein projective $RG$-module is Gorenstein flat when the global dimension of $R$ is finite.
\end{abstract}

\maketitle

\addtocounter{section}{-1}
\section{Introduction}The concept of $G$-dimension for commutative Noetherian rings was introduced by Auslander and Bridger \cite{AB} and has been extended to modules over any ring $R$ through the notion of a Gorenstein projective module by Enochs and Jenda \cite{EJ}. Gorenstein injective modules are defined dually in \cite{EJ}, while Gorenstein flat modules were defined in \cite{EJ2}. The relative homological dimensions based on these modules were defined in \cite{H1}, which is the standard reference for these notions. Projectively coresolved Gorenstein flat modules (PGF modules, for short) were introduced by Saroch and Stovicek \cite{SS}. Over a ring $R$, these modules are the syzygies of the acyclic complexes of projective modules that remain acyclic after applying the functor $I\otimes_R  \_\!\_$ for every injective module $I$. It is clear that PGF modules are Gorenstein flat. As shown in \cite[Theorem 4.4]{SS}, the PGF modules are also Gorenstein projective. Dalezios and Emmanouil \cite{DE} studied the relative homological dimension based on the class of PGF modules. The PGF dimension is a refinement of the ordinary projective 
dimension, whereas the Gorenstein projective dimension is a refinement of the 
PGF dimension. Holm’s metatheorem \cite{Ho} states that every result in classical homological algebra has a counterpart in Gorenstein homological algebra. However, the relation between Gorenstein projective and Gorenstein flat modules is not well understood. For example, projective modules are always flat, but it is not clear whether all Gorenstein projective modules are Gorenstein flat.

In this paper, we define the projectively coresolved Gorenstein flat dimension of a group $G$ over a commutative ring $R$ (PGF dimension of $G$ over $R$, for short) as the PGF dimension of the trivial $RG$-module $R$, which we denote by ${\widetilde{\textrm{Gcd}}_{R}G}$, and we prove that this dimension enjoys all the properties of the Gorenstein cohomological dimension studied by Emmanouil and Talelli in \cite{Em-Ta,ET,ET2}. A central role in the characterization of the finiteness of ${\widetilde{\textrm{Gcd}}_{R}G}$ is played by the characteristic modules of $G$ over $R$ (see Definition \ref{defi}). Characteristic modules were used before to prove many properties of the Gorenstein cohomological dimension $\textrm{Gcd}_R G$ of a group $G$ (see \cite{BDT,Tal}). Many of our results concerning ${\widetilde{\textrm{Gcd}}_{R}G}$ depend on the finiteness of the invariants  $\textrm{sfli}R$ and $\textrm{spli}R$. The invariant $\textrm{spli}R$ has been defined by Gedrich and Gruenberg \cite{GG} as the supremum of the projective lengths (dimensions) of injective left $R$-modules, while the invariant $\textrm{sfli}R$ is defined similarly as the supremum of the flat lengths (dimensions) of injective left $R$-modules. Since for every commutative ring $R$ the finiteness of the invariant $\textrm{sfli}R$ is equivalent with the finiteness of the Gorenstein weak global dimension of $R$ and the finiteness of the invariant $\textrm{spli}R$ is equivalent with the finiteness of the Gorenstein global dimension of the ring $R$ (see Section 1), we generalize all results in \cite{Em-Ta, ET,ET2} which are stated over commutative rings of finite global or weak global dimension in this Gorenstein setting. In particular, we may replace in every statement in \cite{Em-Ta, ET,ET2} the global dimension $\textrm{gl.dim}R$ of $R$ with the invariant $\textrm{spli}R$ and the weak global dimension $\textrm{wgl.dim}R$ of $R$ with the invariant $\textrm{sfli}R$.

The following statement is the first main result of this paper (see Theorem \ref{theorr}); it generalizes \cite[Theorem 1.7]{ET}.

\begin{Theorem}
	Let $G$ be a group and $R$ be a commutative ring such that $\textrm{spli}R<\infty$. The following conditions are equivalent:
	\begin{itemize}
		\item[(i)] ${\widetilde{\textrm{Gcd}}_{R}G}<\infty$.	
		\item[(ii)]${\textrm{Gcd}}_{R}G<\infty$.
		\item[(iii)] There exists a characteristic module $\Lambda$ for $G$ over $R$.
		\item[(iv)] Every $RG$-module has finite PGF dimension.
		\item[(v)] Every $RG$-module has finite Gorenstein projective dimension.
		\item[(vi)] $\textrm{silp}(RG)=\textrm{spli}(RG)< \infty$.
	\end{itemize}
	In this case, we have $\textrm{PGF-dim}_{RG}M=\textrm{Gpd}_{RG}M$ for every $RG$-module $M$.
\end{Theorem}

The second main result concerns approximations of the PGF global dimension of $RG$ over any commutative ring $R$ and any group $G$, and yields also approximations for the Gorenstein global dimension of $RG$ (see Corollaries \ref{corrr} and 2.20).

 \begin{Theorem}
 	Let $R$ be a commutative ring and $G$ be a group. Then, 
 	\begin{itemize}
 		\item[(i)] $ \textrm{max}\{\textrm{PGF-gl.dim}R, {\widetilde{\textrm{Gcd}}_{R}G}\}\leq\textrm{PGF-gl.dim}(RG) \leq {\widetilde{\textrm{Gcd}}_{R}G} + \textrm{PGF-gl.dim}R,$
 		\item[(ii)] $\textrm{max}\{\textrm{Ggl.dim}R, \textrm{Gcd}_{R}G\} \leq\textrm{Ggl.dim}(RG) \leq {{\textrm{Gcd}}_{R}G} + \textrm{Ggl.dim}R.$
 	\end{itemize}
 \end{Theorem}
 
 The following are the properties enjoyed by the PGF dimension of a group $G$ over a commutative ring $R$, which we prove in Sections 3 and 4:
 \begin{itemize}
 	\item[(i)]for every commutative ring $R$ and every group $G$, the dimension ${\widetilde{\textrm{Gcd}}_{R}G}$ vanishes if and only if $G$ is a finite group (see Theorem \ref{theo1});
 	\item[(ii)]if $\textrm{sfli}R$ is finite and $H\subseteq G$ is a subgroup, then ${\widetilde{\textrm{Gcd}}_{R}H}\leq {\widetilde{\textrm{Gcd}}_{R}G}$ (see Proposition \ref{prop38});
 	\item[(iii)]if $\textrm{sfli}R$ is finite and $H\trianglelefteq G$ is a normal subgroup, then ${\widetilde{\textrm{Gcd}}_{R} G}\leq {\widetilde{\textrm{Gcd}}_{R}}H + {\widetilde{\textrm{Gcd}}_{R}}(G/H)$ (see Proposition \ref{prop56});
 	\item[(iv)]if $\textrm{sfli}R$ is finite and $H\subseteq G$ is a finite subgroup with Weyl group $W=N_G(H)/H$, then ${\widetilde{\textrm{Gcd}}_{R}W}\leq{\widetilde{\textrm{Gcd}}_{R}G}$ (see Corollary \ref{cor216});
 	\item[(v)] if the group $G$ is expressed as the union of a continuous ascending chain of subgroups $(G_{\lambda})_{\lambda}$, then ${\widetilde{\textrm{Gcd}}_{R}G}\leq 1 +\textrm{sup}_{\lambda}{\widetilde{\textrm{Gcd}}_{R}G_{\lambda}}$ (see Corollary \ref{cor57}).
 \end{itemize}
 
  \medskip
 Furthermore, we generalize \cite[Theorem 6,4]{Em-Ta} by giving a finiteness criterion for the PGF dimension of a group $G$ over a commutative ring $R$ of finite Gorenstein global dimension which involves only complete cohomology in the following result. We also remove the condition that the commutative ring $R$ is Noetherian (see Theorem \ref{theo43}).
 \begin{Theorem}Let $G$ be a group and $R$ be a commutative ring such that $\textrm{spli}R<\infty$. The following are equivalent:
 	\begin{itemize}
 		\item[(i)] ${\widetilde{\textrm{Gcd}}_{R}G}<\infty$.
 		\item[(ii)] There exists an $R$-split monomorphism of $RG$-modules
 		$\iota: R \rightarrow \Lambda$, where $\Lambda$ is $R$-projective, such that the image of $\iota \in \textrm{Hom}_{RG}(R,\Lambda)=\textrm{H}^{\,0} (G,\Lambda)$ vanishes in the group $\widehat{\textrm{Ext}}^{0}_{RG}(R,\Lambda)={\widehat{\textrm{H}}}^{\,0}(G,\Lambda)$.
 		\item[(iii)]There exists a characteristic module for $G$ over $R$.
 		\item[(iv)] $\textrm{silp}(RG)=\textrm{spli}(RG)< \infty$.
 	\end{itemize}
 \end{Theorem}

 %(see Theorem \ref{theo76})
 %\begin{Theorem}Let $R$ be a commutative ring such that $\textrm{sfli}R <\infty$ and $G$ be a group of type FP$_{\infty}$ over $R$ with ${\widetilde{\textrm{Gcd}}_{R}G}<\infty$. Then, there exists a characteristic module for $G$ over $R$ of type FP. \end{Theorem}

Special cases of groups are also studied. For example, the following theorem determines the PGF dimension of $\textsc{\textbf{lh}}\mathfrak{F}$-groups and yields generalizations of \cite[Theorem 3.1]{Bis} and \cite[Theorem A.1]{ET2} (see Theorem \ref{Theo712}).
 \begin{Theorem}We have ${\widetilde{\textrm{Gcd}}_R G}=\textrm{Gcd}_R G=\textrm{pd}_{RG} B(G,R)$, for every commutative ring $R$ such that $\textrm{spli}R<\infty$ and every $\textsc{\textbf{lh}}\mathfrak{F}$-group $G$.
 \end{Theorem}
 
 It is well known that over a right $\aleph_0$-coherent ring, every countably presented Gorenstein projective $R$-module is PGF and hence is Gorenstein flat (see \cite[Lemma 3.4]{WL}). We do not know any example of a Gorenstein projective module which is not (projectively coresolved) Gorenstein flat. The relation between Gorenstein projective and Gorenstein flat modules is a main open problem in the area. According to Saroch and Stovicek, the notion of a PGF module could serve as an alternative definition of a Gorenstein projective module over any ring (see \cite[Remark on pages 23-24]{SS}). Thus, the following result of the paper is noteworthy (see Theorem \ref{theo94}).
 \begin{Theorem}Let $G$ be a group in $\textsc{\textbf{lh}}\mathfrak{F}$ or of type $\Phi_R$ over a commutative ring $R$ of finite global dimension. Then, the class of Gorenstein projective $RG$-modules coincides with the class of PGF $RG$-modules. Hence, every Gorenstein projective $RG$-module is Gorenstein flat.
 \end{Theorem}
 
 The contents of the paper are as follows. In Section 1 we establish notation, terminology and preliminary results that will be used in the sequel. In Section 2 we provide finiteness criteria for the PGF dimension of a group $G$ over a commutative ring $R$ of finite Gorenstein global dimension (see Theorem \ref{theorr}). A useful tool for our study is this section is the concept of a characteristic module for a group $G$ over $R$ (see Definition \ref{defi}). Moreover, we obtain good estimations for the invariant $\textrm{spli}(RG)$ in terms of the dimension ${\widetilde{\textrm{Gcd}}_{R}G}$ and the invariant $\textrm{spli}R$. We also prove that the invariant $\textrm{spli}(RG)$ is subadditive under group extensions over any commutative ring of finite Gorenstein global dimension. These results yield nice estimations for the PGF and Gorenstein global dimensions $\textrm{PGF-gl.dim}(RG)$ and $\textrm{Ggl.dim}(RG)$. Furthermore, we provide an upper bound for the PGF dimension $\textrm{PGF-dim}_{RG}M$ of a $RG$-module $M$ in terms of PGF dimension $\widetilde{\textrm{Gcd}}_{R}G$ of $G$ over $R$ and the PGF dimension $\textrm{PGF-dim}_{R}M$ of the restricted $k$-module $M$. In this way, we give a PGF analogue of the well-known bound of the projective dimension $\textrm{pd}_{RG}M$ of an $RG$-module $M$, in terms of the cohomological dimension $\textrm{cd}_R G$ of the group $G$ and the projective dimension $\textrm{pd}_R M$ of the restricted $R$-module $M$.

In Section 3 we study the dependence of the PGF dimension of a group $G$ upon the coefficient ring, its behaviour with respect to subgroups and extensions and the subadditivity with respect to group extensions. Our main tool is the existence of a characteristic module obtained from Corollary \ref{cor1}. Moreover, we obtain PGF and Gorenstein projective analogues for Serre's theorem (see \cite[VIII, Theorem 3.1]{Br}).

In Section 4 we study hyperfinite extensions of PGF modules and we show that every hyper-${\tt PGF}(R)$ module is PGF (see Corollary \ref{cor54}). Using this result, for any group $G$ which is expressed as the union of a continuous ascending chain of subgroups $(G_{\lambda})_{\lambda}$, we provide an upper bound in terms of the PGF dimension of its subgroups (see Corollary \ref{cor57}).

In Section 5 we examine the relevance of complete cohomology in the study of modules of finite PGF dimension and we give a finiteness criterion for ${\widetilde{\textrm{Gcd}}_{R}G}$ which involves only complete cohomology (see Theorem \ref{theo43}). Furthermore, we show that the conditions characterizing the finiteness of the PGF dimension of a module $M$, described in \cite[Theorem 3.4(ii),(iii)]{DE}, may be relaxed by removing the assumption that the module $K$ has finite projective dimension in each case and assuming instead that certain elements of complete cohomology groups vanish (see Theorem \ref{Theor61}).

In Section 6 we examine the special case of groups of type FP$_{\infty}$ with finite PGF dimension and show that these groups have characteristic modules of type FP (see Theorem \ref{theo76}). Moreover, we give an analogue of Fel'dman's theorem \cite{Fe} concerning PGF dimensions of groups. In particular, we examine conditions under which the inequality ${\widetilde{\textrm{Gcd}}_{R} G}\leq {\widetilde{\textrm{Gcd}}_{R}}H + {\widetilde{\textrm{Gcd}}_{R}}(G/H)$ of Proposition \ref{prop56} is an equality (see Theorem \ref{theo613} and Corollary \ref{cor614}).

In Section 7 we determine the PGF dimension ${\widetilde{\textrm{Gcd}}_{R}G}$ of an $\textsc{\textbf{lh}}\mathfrak{F}$-group $G$ over a commutative ring of finite Gorenstein global dimension, in terms of the projective dimension of the $RG$-module $B(G,R)$.

In Section 8 we show that every cofibrant $RG$-module is PGF. Using this result, we prove that for every commutative ring $R$ of finite global dimension and every group $G$ which is $\textsc{\textbf{lh}}\mathfrak{F}$ or of type $\Phi_R$, every Gorenstein projective $RG$-module is Gorenstein flat. This result is interesting, since it is not known whether all Gorenstein projective modules are Gorenstein flat over an arbitrary ring.

\smallskip

\noindent\textbf{Conventions.} All rings are assumed to be associative and unital and all ring homomorphisms will be unit preserving. Unless otherwise specified, all modules will be left $R$-modules.

\section{Preliminaries}
In this section we collect certain notions and preliminary results that will be used throughout the paper.

\subsection{Gorenstein projective, Gorenstein flat and PGF modules.}
An acyclic complex $\textbf{P}$ of projective modules is said to be a complete 
projective resolution if the complex of abelian groups $\mbox{Hom}_R(\textbf{P},Q)$
is acyclic for every projective module $Q$. Then, a module is Gorenstein 
projective if it is a syzygy of a complete projective resolution. The 
Gorenstein projective dimension $\mbox{Gpd}_RM$ of a module $M$ is the 
length of a shortest resolution of $M$ by Gorenstein projective modules. 
If no such resolution of finite length exists, then we write 
$\mbox{Gpd}_RM = \infty$. If $M$ is a module of finite projective 
dimension, then $M$ has finite Gorenstein projective dimension as 
well and $\mbox{Gpd}_RM = \mbox{pd}_RM$.

An acyclic complex $\textbf{F}$ of flat modules is said to be a complete flat
resolution if the complex of abelian groups $I \otimes_R \textbf{F}$ is acyclic 
for every injective right module $I$. Then, a module is Gorenstein 
flat if it is a syzygy of a complete flat resolution. We let ${\tt GFlat}(R)$ 
be the class of Gorenstein flat modules. The Gorenstein flat dimension 
$\mbox{Gfd}_RM$ of a module $M$ is the length of a shortest resolution 
of $M$ by Gorenstein flat modules. If no such resolution of finite length
exists, then we write $\mbox{Gfd}_RM = \infty$. If $M$ is a module of 
finite flat dimension, then $M$ has finite Gorenstein flat dimension as
well and $\mbox{Gfd}_RM = \mbox{fd}_RM$.

The notion of a projectively coresolved Gorenstein flat module (PGF-module, for short) is introduced by Saroch and Stovicek \cite{SS}. A PGF module is a syzygy of an acyclic complex of projective modules $\textbf{P}$, which is such that the complex of abelian groups $I \otimes_R \textbf{P}$ is acyclic for every injective module $I$. It is clear that the class ${\tt PGF}(R)$ of PGF modules is contained in ${\tt GFlat}(R)$. The inclusion ${\tt PGF}(R) \subseteq {\tt GProj}(R)$ is proved in \cite[Theorem 4.4]{SS}. Moreover, the class of PGF $R$-modules, is closed under extensions, direct sums, direct summands and kernels of epimorphisms. The PGF dimension $\mbox{PGF-dim}_RM$ of a module $M$ is the length of a shortest resolution of $M$ by PGF modules. If no such resolution of finite length exists, then we write $\mbox{PGF-dim}_RM = \infty$. If $M$ is a module of finite projective dimension, then $M$ has finite PGF dimension as well and $\mbox{PGF-dim}_RM = \mbox{pd}_RM$ (see \cite[Proposition 2.2, Corollary 3.7(i)]{DE}).

\begin{Remark}\label{r1}\rm Invoking \cite[Propositions 2.4 and 3.6]{DE}, it follows easily that for every short exact sequence of $R$-modules $0\rightarrow M'\rightarrow M \rightarrow M'' \rightarrow 0$, we have $\textrm{PGF-dim}_RM''\leq \textrm{max}\{1+\textrm{PGF-dim}_R M',\textrm{PGF-dim}_R M \}$.\end{Remark}
	
\begin{Lemma}\label{lem63}
	Let $R$ be a ring and consider an integer $n\geq 1$ and an exact sequence of $R$-modules $$0\rightarrow M_n \rightarrow \cdots \rightarrow M_1 \rightarrow M_0 \rightarrow M \rightarrow 0.$$ Then, $\textrm{PGF-dim}_{R}M\leq \textrm{max}\{i+\textrm{PGF-dim}_{R}M_i: i=0,\dots,n\}$.
\end{Lemma}

\begin{proof}We proceed by induction on $n\geq 1$. The case $n=1$ follows from Remark \ref{r1}. We assume now that $n>1$ and let $M'=\textrm{Im}(M_1 \rightarrow M_0)$. Applying the induction hypothesis on the exact sequence $$0\rightarrow M_n \rightarrow \cdots \rightarrow M_1 \rightarrow M' \rightarrow 0,$$ we obtain that $\textrm{PGF-dim}_{R}M'\leq \textrm{max}\{i-1+\textrm{PGF-dim}_{R}M_i: i=1,\dots,n\}$. Using again Remark \ref{r1} and the short exact sequence $0\rightarrow M' \rightarrow M_0 \rightarrow M \rightarrow 0$, we get $\textrm{PGF-dim}_{R}M\leq \textrm{max}\{1+\textrm{PGF-dim}_{R}M',\textrm{PGF-dim}_{R}M_0\}\leq \textrm{max}\{i+\textrm{PGF-dim}_{R}M_i: i=0,\dots,n\}$, as needed. 
\end{proof}

\subsection{Gedrich-Gruenberg invariants and Gorenstein global dimensions}The invariants $\textrm{silp}R$, $\textrm{spli}R$ were defined by Gedrich and Gruenberg in \cite{GG} as the supremum of the injective lengths (dimensions) of projective modules and the supremum of the projective lengths (dimensions) of injective modules, respectively. The invariant $\textrm{sfli}R$ is defined similarly as the supremum of the flat lengths (dimensions) of injective modules. It is clear that $\textrm{sfli}R\leq \textrm{spli}R$ for every ring $R$. The Gorenstein global dimension $\textrm{Ggl.dim}R$ of a ring $R$ is defined as the supremum of the Gorenstein projective dimensions of $R$-modules, while the Gorenstein weak global dimension $\textrm{Gwgl.dim}R$ of a ring $R$ is defined as the supremum of the Gorenstein flat dimensions of $R$-modules. The PGF global dimension of a ring $R$ defined recently in \cite{DE} as the supremum of the PGF dimensions of $R$-modules. Considering the case where the ring $R$ is commutative, \cite[Corollary 5.4]{DE} yields the inequality $\textrm{silp}R\leq\textrm{spli}R$, with equality if $\textrm{spli}R<\infty$. Therefore, for every commutative ring $R$, invoking \cite[Theorem 4.1]{Emm3}, we infer that the invariant $\textrm{spli}R$ is finite if and only if the Gorenstein global dimension $\textrm{Ggl.dim}R$ is finite, and moreover $\textrm{Ggl.dim}R=\textrm{spli}R$. Furthermore, for every commutative ring $R$, invoking \cite[Theorem 2.4]{CET}, we infer that the invariant $\textrm{sfli}R$ is finite if and only if the Gorenstein weak global dimension $\textrm{Gwgl.dim}R$ is finite, and moreover $\textrm{Gwgl.dim}R=\textrm{sfli}R$.

\begin{Lemma}\label{lemgl}Let $S$ be a ring which is isomorphic with its opposite $S^{\textrm{op}}$. Then, $\textrm{PGF-gl.dim}S=\textrm{Ggl.dim}S.$
\end{Lemma}

\begin{proof}Since the PGF dimension bounds the Gorenstein projective dimension, we have $\textrm{Ggl.dim}S\leq \textrm{PGF-gl.dim}S$. Thus, it remains to show that $\textrm{PGF-gl.dim}S\leq \textrm{Ggl.dim}S$. For that, it suffices to assume that $\textrm{Ggl.dim}S<\infty$. Then, $\textrm{Ggl.dim}S=\textrm{spli}S<\infty$, by \cite[Theorem 4.1]{Emm3}. Since $S\cong {S}^{\textrm{op}}$ and $\textrm{sfli}S\leq \textrm{spli}S<\infty$, invoking \cite[Theorem 5.1]{DE}, we infer that $\textrm{PGF-gl.dim}S=\textrm{Ggl.dim}S$, as needed.
\end{proof}

\subsection{Group rings.}Let $R$ be a commutative ring, $G$ be a group and consider the associated group ring $RG$. The standard reference for group cohomology is \cite{Br}. The anti-isomorphism of $RG$, which is induced by the map $g\rightarrow g^{-1}$, $g\in G$, enables us to view every right $RG$-module $M$ as a left $RG$-module $M$, and hence $kG\cong {(kG)}^{\textrm{op}}$. Using the diagonal action of the group $G$, the tensor product $M\otimes_R N$ of two $RG$-modules is also an $RG$-module; we define $g\cdot (x \otimes y)=gx \otimes gy \in M\otimes_R N$ for every $g\in G$, $x\in M$ and $y\in N$. We note that for every projective $RG$-module $M$ and every $R$-projective $RG$-module $N$, the diagonal $RG$-module $M\otimes_R N$ is also projective.

An $RG$-module $M$ is said to be of type FP$_{\infty}$ (respectively, of type FP) if M admits a projective resolution $\cdots \rightarrow P_n \rightarrow P_{n-1}\rightarrow \cdots \rightarrow P_0 \rightarrow M \rightarrow 0,$ where $P_i$ is finitely generated projective for every $i\geq 0$ (respectively, $P_i$ is finitely generated projective and vanish for $n>>0$). The group $G$ is said to be of type FP$_{\infty}$ (respectively, of type FP) over $R$ if the trivial $RG$-module $R$ is of type FP$_{\infty}$ (respectively, of type FP).

\smallskip
The following lemmata concern properties of PGF modules over group rings that will be used in the sequel.

\begin{Lemma}\label{lem46}
	Let $R$ be a commutative ring, $G$ be a group and $H$ be a subgroup of $G$.
	\begin{itemize}
		\item[(i)] For every PGF $RH$-module $M$, the $RG$-module $\textrm{Ind}^G_H M$ is also PGF.
		
		\item[(ii)] For every $RG$-modules $M$, $N$ such that $M$ is projective and $N$ is PGF as $RH$-module, the $RG$-module $M\otimes_{RH}N$ is PGF.
		
		\item[(iii)] For every $RH$-module, $\textrm{PGF-dim}_{RG}(\textrm{Ind}_H^G M) \leq \textrm{PGF-dim}_{RH} M$.
		%\item[(iv)] If the subgroup $H$ of $G$ is of finite index, then for every PGF $RG$-module, the $RH$-module $\textrm{Res}^G_H M$ is PGF as well.
	\end{itemize}  
\end{Lemma}

\begin{proof} (i) Let $M$ be a PGF $RH$-module. Then, there exists an acyclic complex of projective $RH$-modules $$\textbf{P}=\cdots \rightarrow P_{2}\rightarrow P_1\rightarrow P_0 \rightarrow P_{-1}\rightarrow \cdots,$$ such that $M=\textrm{Im}(P_1 \rightarrow P_0)$ and the complex $I\otimes_{RH}\textbf{P}$ is exact, whenever $I$ is an injective $RH$-module. Then, the induced complex $$\textrm{Ind}^G_H\textbf{P}=\cdots \rightarrow\textrm{Ind}^G_H P_2 \rightarrow\textrm{Ind}^G_H P_1\rightarrow\textrm{Ind}^G_H P_0 \rightarrow\textrm{Ind}^G_H P_{-1}\rightarrow \cdots,$$ is an acyclic complex of projective $RG$-modules and has the $RG$-module $\textrm{Ind}^G_H M$ as syzygy. Moreover, for every injective $RG$-module $I$, the restricted $RH$-module $I|_H$ is also injective. Thus, the isomorphism of complexes $I\otimes_{RG}\textrm{Ind}^G_H \textbf{P} \cong I|_H\otimes_{RH}\textbf{P}$ implies that the $RG$-module $\textrm{Ind}^G_H M$ is also PGF.
	
	(ii) Since the class of PGF modules is closed under direct sums and direct summands, it suffices to assume that $M=RG$. Then, (i) implies that the $RG$-module $M\otimes_{RH}N$ is PGF.
	
	(iii) It suffices to assume that $\textrm{PGF-dim}_{RH} M=n$ is finite. Then, there exist PGF $RH$-modules $P_0, P_1, \dots ,P_n$ and an exact sequence of $RH$-modules $$0\rightarrow P_n \rightarrow \cdots \rightarrow P_1 \rightarrow P_0 \rightarrow M \rightarrow 0.$$ By (i) we have already prove, the $RG$-modules $\textrm{Ind}^G_H P_i$ are PGF, for every $i=0,\dots ,n$. Thus, the induced exact sequence of $RG$-modules $$0\rightarrow\textrm{Ind}^G_H P_n \rightarrow \cdots \rightarrow\textrm{Ind}^G_H P_1 \rightarrow\textrm{Ind}^G_H P_0 \rightarrow\textrm{Ind}^G_H M \rightarrow 0,$$ yields $\textrm{PGF-dim}_{RG}(\textrm{Ind}_H^G M) \leq n$.\end{proof}
	
\begin{Lemma}\label{lemZ}Let $R$ be a commutative ring, $G$ be a group and $H$ be a normal subgroup of $G$. Then, for every projective $R[G/H]$-module $M$ and every $RG$-module $N$ which is projective as $RH$-module, the $RG$-module $M\otimes_R N$ is projective.
\end{Lemma}	

\begin{proof}It suffices to assume that $M=R[G/H]$. Then, $M\otimes_R N=R[G/H]\otimes_R N\cong \textrm{Ind}^G_H \textrm{Res}^G_H N$ which is a projective $RG$-module (see \cite[Proposition 5.6(a)]{Br}). \end{proof}

\begin{Lemma}\label{lemmm22}Let $R$ be a commutative ring, $G$ be a group and $H$ be a subgroup of $G$. Then, $\textrm{spli}(RH)\leq \textrm{spli}(RG)$.
\end{Lemma}

\begin{proof}It suffices to assume that $\textrm{spli}(RG)=n<\infty$. Let $I$ be an injective $RH$-module. Since the $RG$-module $\textrm{Coind}^G_H I$ is injective, we have $\textrm{pd}_{RG}\textrm{Coind}^G_H I\leq n$ and hence $\textrm{pd}_{RH}\textrm{Coind}^G_H I\leq n$. As the $RH$-module $I$ is a direct $RH$-summand of $\textrm{Coind}^G_H I$, we obtain that $\textrm{pd}_{RH} I\leq n$. Consequently, we have $\textrm{spli}(kH)\leq n$. 
\end{proof}	

One of our main tools in this paper, is the concept of a characteristic module for a group $G$ over a commutative ring $R$.

\begin{Definition}\label{defi}Let $R$ be a commutative ring and $G$ be a group. A characteristic module for $G$ over $R$ is defined to be an $R$-projective $RG$-module $\Lambda$ with $\textrm{pd}_{RG}\Lambda <\infty$, which admits an $R$-split $RG$-linear monomorphism $i: R \rightarrow \Lambda$.\end{Definition}

 Characteristic modules were used to prove many properties of the Gorenstein cohomological dimension $\textrm{Gcd}_R G$ of a group $G$ (see \cite{BDT,Tal}).

\subsection{$\textsc{\textbf{lh}}\mathfrak{F}$-groups and groups of type $\Phi_R$} The class $\textsc{\textbf{h}}\mathfrak{F}$ was defined by Kropholler in \cite{Kr}. This is the smallest class of groups, which contains the class $\mathfrak{F}$ of finite groups and is such that whenever a group $G$ admits a finite dimensional
contractible $G$-CW-complex with stabilizers in $\textsc{\textbf{h}}\mathfrak{F}$, then we also have $G\in \textsc{\textbf{h}}\mathfrak{F}$. The class $\textsc{\textbf{lh}}\mathfrak{F}$ consists of those groups, all of whose finitely generated subgroups are in $\textsc{\textbf{h}}\mathfrak{F}$. All soluble groups, all groups of finite virtual cohomological dimension and all automorphism groups of Noetherian modules over a commutative ring are $\textsc{\textbf{lh}}\mathfrak{F}$-groups. The class $\textsc{\textbf{lh}}\mathfrak{F}$ is closed under extensions, ascending unions, free products with amalgamation and HNN extensions.

A group $G$ is said to be of type $\Phi_R$ if it has the property that for every $RG$-module $M$, $\textrm{pd}_{RG}M<\infty$ if and only if $\textrm{pd}_{RH}M<\infty$ for every finite subgroup $H$ of $G$. These groups were defined over $\mathbb{Z}$ in \cite{Ta}. Over a commutative ring $R$ of finite global dimension, every group of finite virtual cohomological dimension and every group which acts on a tree with finite stabilizers is of type $\Phi_R$ (see \cite[Corollary 2.6]{MS}).

Let $B(G,R)$ be the $RG$-module which consists of all functions from $G$ to $R$ whose image is a finite subset of $R$. The $RG$-module $B(G,R)$ is $R$-free and $RH$-free for every finite subgroup $H$ of $G$. For every element $\lambda \in R$, the constant function $\iota(\lambda)\in B(G,R)$ with value $\lambda$ is invariant under the action of $G$. The map $\iota: R \rightarrow B(G,R)$ which is defined in this way is then $RG$-linear and $R$-split. Indeed, for every fixed element $g\in G$, there exists an $R$-linear splitting for $\iota$ by evaluating functions at $g$. Moreover, the cokernel $\overline{B}(G,R)$ of $\iota$ is $R$-free (see \cite[Lemma 3.3]{Kr2} and \cite[Lemma 3.4]{BC}). We note that the $RG$-module $B(G,R)$ is a candidate for a characteristic module for any group $G$ over any commutative ring $R$.

\section{Modules of finite PGF dimension}
For any unital commutative ring $R$ and any group $G$, we denote by ${\widetilde{\textrm{Gcd}}_{R}G}$ the PGF dimension of $G$ over $R$, i.e. the PGF dimension of the trivial $RG$ module $R$. In this section, using characteristic modules for a group $G$ over a commutative ring $R$ of finite Gorenstein global dimension, we give finiteness criteria for the PGF dimension of the group $G$. Moreover, we provide good estimations for the Gorenstein global dimension of $RG$, in terms of the Gorenstein global dimension of the base ring $R$ and the dimensions ${\widetilde{\textrm{Gcd}}_{R}G}$, $\textrm{Gcd}_RG$, over any commutative ring $R$ and any group $G$. Finally, we construct an upper bound for the PGF dimension of an $RG$-module $M$, in terms of the PGF dimension of the group $G$ and and the PGF-dimension of the restricted $R$-module $M$.

\subsection{Finiteness criteria for the PGF dimension of groups}The goal of this subsection is to provide finiteness criteria for the PGF dimension ${\widetilde{\textrm{Gcd}}_{R}G}$ of a group $G$ over a commutative ring $R$ of finite Gorenstein global dimension.

\begin{Lemma}\label{lem1}
	Let $R$ be a commutative ring such that $\textrm{sfli}R<\infty$ and $G$ be a group. Then every PGF $RG$-module is PGF as $R$-module.
\end{Lemma}

\begin{proof}
	Let $M$ be a PGF $RG$-module. Then, there exists an acyclic complex of projective $RG$-modules $$\textbf{P}= \cdots \rightarrow P_2\rightarrow P_1 \rightarrow P_0 \rightarrow P_{-1}\rightarrow \cdots$$ which remains acyclic after the application of the functor $I\otimes_{RG} \_\!\_$, for every injective $RG$-module $I$, and $M=\textrm{Im}(P_1 \rightarrow P_0)$. Since every projective $RG$-module is also a projective $R$-module, the restriction of $\textbf{P}$ yields an acyclic complex of projective $R$-modules. Moreover, the condition $\textrm{sfli}R<\infty$ implies that every acyclic complex of projective $R$-modules, remains acyclic after the application of the functor $I\otimes_{R} \_\!\_$, for every injective $R$-module $I$. Thus, $M$ is PGF as $R$-module.
\end{proof}

\begin{Proposition}\label{Prop} Let $R$ be a commutative ring such that $\textrm{sfli}R<\infty$ and consider a group $G$ and an $RG$-module $M$ of finite PGF dimension. If the projective dimension $\textrm{pd}_R M$ is finite, then there exists an $R$-split $RG$-exact sequence $0\rightarrow M \rightarrow \Lambda \rightarrow L \rightarrow 0$, where $L$ is an $R$-projective $RG$-module and $\textrm{PGF-dim}_{RG}M=\textrm{pd}_{RG}\Lambda$.
\end{Proposition}

\begin{proof}
	Let $\textrm{PGF-dim}_{RG}M=n<\infty$. By \cite[Theorem 3.4]{DE} there exists a short exact sequence of $RG$-modules $0\rightarrow M \rightarrow \Lambda \rightarrow L\rightarrow 0$, where $L$ is a PGF $RG$-module and $\textrm{pd}_{RG}\Lambda=n$. Since $L$ is a PGF $RG$-module, invoking Lemma \ref{lem1} we infer that $L$ is PGF as $R$-module. Thus, by \cite[Proposition 3.6]{DE} we have $\textrm{Ext}^1_{R}(L,M)=0$ and the exact sequence $0\rightarrow M\rightarrow \Lambda\rightarrow L\rightarrow 0$ is $R$-split. Moreover the finiteness of $\textrm{pd}_{R}M$ and $\textrm{pd}_{R}\Lambda$ yields the finiteness of $\textrm{pd}_{R}L$ and so $L$ is $R$-projective; see \cite[Corollary 3.7(i)]{DE}. \end{proof}

\begin{Corollary}\label{cor1}
	Let $R$ be a commutative ring such that $\textrm{sfli}R<\infty$ and $G$ be a group such that ${\widetilde{\textrm{Gcd}}_{R}G}<\infty$. Then, there exists a characteristic module $\Lambda$ for $G$ over $R$ and $\textrm{pd}_{RG}\Lambda ={\widetilde{\textrm{Gcd}}_{R}G}$.
\end{Corollary}

\begin{Remark}\label{rem1}\rm Let $M$, $N$ be two $RG$-modules. Then, the tensor product $M\otimes_R N$ is also an $RG$-module with the diagonal action of the group $G$. By \cite[III Corollary 5.7]{Br} we have that if $M$ is a projective $RG$-module and $N$ is an $R$-projective $RG$-module, then the $RG$-module $M\otimes_R N$ is also projective. Thus, if $\textbf{P}$ is a projective resolution of an $RG$-module $M$ and $N$ is an $R$-projective $RG$-module, then $\textbf{P} \otimes_R N$ is a projective resolution of the $RG$-module $M\otimes_R N$ and $\textrm{pd}_{RG}(M\otimes_R N)\leq \textrm{pd}_{RG}M$.
\end{Remark}

\begin{Proposition}\label{prop1} Let $R$ be a commutative ring and $G$ be a group such that there exists an $R$-split monomorphism of $RG$-modules $\iota: R \rightarrow \Lambda$, where $\Lambda$ is $R$-projective and $\textrm{pd}_{RG} \Lambda <\infty$. Then, for every $R$-projective $RG$-module $M$ we have $\textrm{PGF-dim}_{RG}M \leq \textrm{pd}_{RG} \Lambda$.
\end{Proposition}

\begin{proof} Let $\textrm{pd}_{RG} \Lambda =n$ and $M$ be an $R$-projective $RG$-module. We consider an $RG$-projective resolution $$\textbf{P}=\cdots \rightarrow P_2 \rightarrow P_1 \rightarrow P_0 \rightarrow M \rightarrow 0$$ of $M$ and set $M_i=\textrm{Im}(P_i \rightarrow P_{i-1})$, $i\geq 0$, where $M_0=M$, the corresponding syzygy modules. Since every projective $RG$-module is PGF, it suffices to prove that the $RG$-module $M_n$ is PGF. Indeed, the $RG$-exact sequence $$0\rightarrow M_n \rightarrow P_{n-1} \rightarrow \cdots \rightarrow P_0 \rightarrow M \rightarrow 0,$$ will then be a PGF resolution of $M$ of length $n$. As $M$ is $R$-projective, the restricted exact sequence $\textbf{P}$ is $R$-split implying the existence of the induced $RG$-exact sequence $$\textbf{P} \otimes_R \Lambda = \cdots \rightarrow P_1\otimes_R \Lambda \rightarrow P_0\otimes_R \Lambda \rightarrow M\otimes_R \Lambda \rightarrow 0$$ with diagonal action. Since the $RG$-module $\Lambda$ is $R$-projective, the exact sequence $\textbf{P} \otimes_R \Lambda$ is a projective resolution of the $RG$-module $M\otimes_R \Lambda$ and $\textrm{pd}_{RG}(M\otimes_R \Lambda)\leq n$ (see Remark \ref{rem1}). Moreover, the corresponding $i$-th syzygies are the modules $M_i \otimes_R \Lambda$, $i\geq 0$. Thus, the $RG$-module $M_n \otimes_R \Lambda$ is projective, and a similar argument implies that the diagonal $RG$-module $M_n \otimes_R N \otimes_R \Lambda$ is also projective for every $R$-projective $RG$-module $N$. Let $L=\textrm{Coker}\iota$ and consider the $R$-split short exact sequence of $RG$-modules $0\rightarrow R \xrightarrow{\iota} \Lambda \rightarrow L \rightarrow 0$. Then, for every $j\geq 0$, we obtain a short exact sequence of $RG$-modules of the form $$0\rightarrow M_n \otimes_R \otimes L^{\otimes j} \rightarrow M_n \otimes_R L^{\otimes j} \otimes_R \Lambda \rightarrow M_n \otimes_R \otimes L^{\otimes j+1}\rightarrow 0,$$ where we denote by $L^{\otimes j}$ the $j$-th tensor power of $L$ over $R$. Since the $RG$-module $L$ is $R$-projective, we obtain that the $RG$-module $L^{\otimes j}$ is also $R$-projective for every $j\geq 0$. Thus, the diagonal $RG$-modules $M_n \otimes_R L^{\otimes j}\otimes_R \Lambda$ are projective for every $j\geq 0$. The splicing of the above short exact sequences yields the exact sequence $$0\rightarrow M_n \xrightarrow{\eta} M_n \otimes_R \Lambda \rightarrow M_n \otimes_R L\otimes_R \Lambda \rightarrow M_n \otimes_R L^{\otimes 2}\otimes_R \Lambda \rightarrow \cdots .$$ Splicing now the latter exact sequence with the projective resolution $$\cdots \rightarrow P_{n+2}\rightarrow P_{n+1}\rightarrow P_n \xrightarrow{\epsilon} M_n \rightarrow 0$$ of $M_n$, we obtain an acyclic complex of projective $RG$-modules
	$$\mathfrak{P}=\cdots \rightarrow P_{n+2}\rightarrow P_{n+1}\rightarrow P_n \xrightarrow{\eta \epsilon}  M_n \otimes_R \Lambda \rightarrow M_n \otimes_R L\otimes_R \Lambda \rightarrow M_n \otimes_R L^{\otimes 2}\otimes_R \Lambda \rightarrow \cdots ,$$ which has syzygies the $RG$-modules $(M_i)_{i\geq n}$ and $(M_n \otimes_R L^{\otimes j})_{j\geq 1}$. In order to prove that the $RG$-module $M_n$ is PGF, it suffices to show that the complex $ I\otimes_{RG} \mathfrak{P}$ is acyclic for every injective $RG$-module $I$. Let $I$ be an injective $RG$-module. Then, the $R$-split short exact sequence of $RG$-modules $0\rightarrow R \xrightarrow{\iota} \Lambda \rightarrow L \rightarrow 0$ induces the short exact sequence of $RG$-modules (with diagonal action) $0\rightarrow I \rightarrow \Lambda\otimes_R I \rightarrow L\otimes_R I \rightarrow 0$. The injectivity of the $RG$-module $I$ implies that the latter short exact sequence is $RG$-split. Thus, it suffices to show the acyclicity of the complex $(\Lambda\otimes_R I)\otimes_{RG}\mathfrak{P}$. Since the $RG$-module $\Lambda$ is $R$-projective, the complex of $RG$-modules (with diagonal action) $\mathfrak{P}\otimes_R \Lambda$ is acyclic with syzygies the projective $RG$-modules $(M_i \otimes_R \Lambda)_{i\geq n}$ and $(M_n \otimes_R L^{\otimes j} \otimes_R \Lambda)_{j\geq 1}$. Thus, the complex $\mathfrak{P}\otimes_R \Lambda$ is contractible and hence the complex $(\mathfrak{P}\otimes_R \Lambda)\otimes_{RG}I\cong (\Lambda\otimes_R I)\otimes_{RG}\mathfrak{P}$ is acyclic. 
	%Then, the isomorphisms $(\Lambda \otimes_R I)\otimes_{RG} \mathfrak{P}\cong \Lambda \otimes_R (I\otimes_{RG} \mathfrak{P})\cong (I\otimes_{RG} \mathfrak{P})\otimes_R \Lambda \cong I\otimes_{RG} (\mathfrak{P}\otimes_R \Lambda)$ yield the acyclicity of the complex $(\Lambda \otimes_R I)\otimes_{RG} \mathfrak{P}$.
\end{proof}

\begin{Corollary}\label{cor2}Let $R$ be a commutative ring and $G$ be a group such that there exists a characteristic module $\Lambda$ for $G$ over $R$. Then, ${\widetilde{\textrm{Gcd}}_{R}G}\leq \textrm{pd}_{RG} \Lambda$.
\end{Corollary}

 A characteristic module for G over $R$ may not always exist and, if it exists, it is certainly not unique. However, the projective dimension of any characteristic module for
$G$ over $R$ is uniquely determined by the pair $(R, G)$.

\begin{Corollary}\label{cor37}Let $R$ be a commutative ring and $G$ be a group. Then:
	\begin{itemize}
		\item[(i)]If $\Lambda, \Lambda'$ are two characteristic modules for $G$ over $R$, then $\textrm{pd}_{RG}\Lambda =\textrm{pd}_{RG}\Lambda '$.
		\item[(ii)]If there exists a characteristic module $\Lambda$ for $G$ over $R$, then $G$ has finite PGF dimension over $R$ and ${\widetilde{\textrm{Gcd}}_{R}G}\leq \textrm{pd}_{RG} \Lambda$.
		\item[(iii)]If $\textrm{sfli}R <\infty$ and ${\widetilde{\textrm{Gcd}}_{R}G}<\infty$, then there exists a characteristic module $\Lambda$ for $G$ over $R$ with $\textrm{pd}_{RG} \Lambda={\widetilde{\textrm{Gcd}}_{R}G}$.
	\end{itemize}
\end{Corollary}

\begin{proof}(i) Let $\Lambda, \Lambda'$ be two characteristic modules for $G$ over $R$. Then $\textrm{pd}_{RG}\Lambda <\infty$ and hence $\textrm{PGF-dim}_{RG}\Lambda =\textrm{pd}_{RG}\Lambda$ (see \cite[Corollary 3.7(i)]{DE}). Proposition \ref{prop1} for the characteristic module $\Lambda'$ and the $R$-projective $RG$-module $\Lambda$ yields $\textrm{pd}_{RG}\Lambda=\textrm{PGF-dim}_{RG}\Lambda\leq \textrm{pd}_{RG}\Lambda '$. Reversing the roles of $\Lambda$ and $\Lambda '$, we obtain the inequality $\textrm{pd}_{RG}\Lambda'\leq \textrm{pd}_{RG}\Lambda$. We conclude that $\textrm{pd}_{RG}\Lambda=\textrm{pd}_{RG}\Lambda'$.

Assertions (ii) and (iii) are precisely the Corollaries \ref{cor2} and \ref{cor1} respectively.
\end{proof}

\begin{Corollary}\label{cor38}Let $R$ be a commutative ring such that $\textrm{sfli}R<\infty$ and $G$ a group. Then, $G$ has finite PGF dimension over $R$ if and only if there exists a characteristic module $\Lambda$ for $G$ over $R$. In that case, we have $\textrm{pd}_{RG} \Lambda={\widetilde{\textrm{Gcd}}_{R}G}$.
\end{Corollary}

\begin{Proposition}\label{prop2} Let $R$ be a commutative ring and $G$ be a group. Then, $$\textrm{silp}(RG)\leq \textrm{spli}(RG) \leq {\widetilde{\textrm{Gcd}}_{R}G} + \textrm{spli}R.$$
\end{Proposition}

\begin{proof}Since $RG\cong {(RG)}^{\textrm{op}}$, the left inequality is a direct consequence of \cite[Corollary 5.4]{DE}. For the right inequality, it suffices to assume that both ${\widetilde{\textrm{Gcd}}_{R}G}=n$ and $\textrm{spli}R=m$ are finite. Then, Corollary \ref{cor1} implies that there exists an $R$-split monomorphism of $RG$-modules $0\rightarrow R \rightarrow \Lambda$, where $\Lambda$ is an $R$-projective $RG$-module and $\textrm{pd}_{RG}\Lambda =n$. Let $I$ be an injective $RG$-module. It follows that there exists an induced monomorphism of $RG$ modules (with diagonal action) $0\rightarrow I \rightarrow \Lambda\otimes_{R}I$. The injectivity of $I$ yields that the latter monomorphism is $RG$-split and hence it suffices to prove that $\textrm{pd}_{RG}(\Lambda\otimes_{R}I)\leq n+m$. Indeed, let $P_*$ be a projective resolution of the $RG$-module $I$ and $I_m=\textrm{Im}(P_m \rightarrow P_{m-1})$ be the corresponding $m$-syzygy. Since $\textrm{spli}R=m$, we obtain an exact sequence of $RG$-modules $$0\rightarrow I_m \rightarrow P_{m-1} \rightarrow \cdots \rightarrow P_0 \rightarrow I \rightarrow 0$$ where $I_m,P_{m-1},\dots ,P_0$ are all projective as $R$-modules. The projectivity of the $R$-module $\Lambda$ yields an induced exact sequence of $RG$-modules with diagonal action $$0\rightarrow  \Lambda\otimes_R I_m  \rightarrow  \Lambda \otimes_R P_{m-1} \rightarrow \cdots \rightarrow  \Lambda\otimes_R P_0 \rightarrow \Lambda\otimes_R I \rightarrow 0.$$ Since $\textrm{pd}_{RG}(\Lambda\otimes_R K)\leq \textrm{pd}_{RG}\Lambda$ for every $R$-projective $RG$-module $K$ (see Remark \ref{rem1}), the exact sequence above implies that $\textrm{pd}_{RG}(\Lambda\otimes_{R}I)\leq n+m$.\end{proof}

	\begin{Proposition}\label{newpropara}Let $R$ be a commutative ring such that $\textrm{sfli}R<\infty$ and $G$ be a group. Then:
		\begin{itemize}
			\item[(i)]Every Gorenstein projective $RG$-module is Gorenstein projective as $R$-module.
			\item[(ii)]If $\textrm{Gcd}_{R}G<\infty$, then there exists a characteristic module $\Lambda$ for $G$ over $R$ and $\textrm{pd}_{RG}\Lambda =\textrm{Gcd}_{R}G$.
		\end{itemize}
	\end{Proposition}
	
	\begin{proof}(i) Let $M$ be a Gorenstein projective $RG$-module. Then $M$ is a syzygy of a suitable acyclic complex of projective $RG$-modules \textbf{P}. Since every $RG$-projective module is also $R$-projective, viewing \textbf{P} as a complex of $R$-projective modules, the condition $\textrm{sfli}R<\infty$ implies that the $R$-module $M$ is PGF. Thus, $M$ is a Gorenstein projective $R$-module.
		
		(ii) The assumption of the finiteness of the Gorenstein cohomological dimension over $R$ yields the existence of a short exact sequence of $RG$-modules $$0\rightarrow R \rightarrow \Lambda \rightarrow L \rightarrow 0,$$ where $\textrm{pd}_{RG}\Lambda= \textrm{Gcd}_{R}G$ and $L$ is Gorenstein projective. Using (i) above, we obtain that $L$ is Gorenstein projective as $R$-module and hence $\textrm{Ext}^1_R(L,R)=0$ (see \cite[Theorem 2.20]{Ho}). Thus the short exact sequence $0\rightarrow R \rightarrow \Lambda \rightarrow L \rightarrow 0$ is $R$-split. Since $\textrm{pd}_{R}\Lambda\leq \textrm{pd}_{RG}\Lambda<\infty$, we have $\textrm{pd}_{R}L<\infty$. Moreover $L$ is Gorenstein projective as $R$-module. Hence $L$ is $R$-projective and $\Lambda$ is $R$-projective as well. We conclude that $\Lambda$ is a characteristic module for $G$ over $R$.
	\end{proof}
	
	\begin{Remark}\label{Remarkara} \rm We note that the condition $\textrm{sfli}R<\infty$ in Proposition \ref{newpropara}, relaxes the finiteness of the weak global dimension of $R$ in \cite[Corollary 1.3]{ET}.\end{Remark}
		
	%\begin{Proposition}\label{propp213}Let $R$ be a commutative ring such that $\textrm{sfli}R<\infty$ and $G$ be a group. Then, $$\textrm{spli}RG \leq {\textrm{Gcd}}_{R}G + \textrm{spli}R.$$\end{Proposition}
	
	%\begin{proof}Using Proposition \ref{newpropara}, a similar argument as in the proof of Proposition \ref{prop2} yields the inequality  for every commutative ring $R$.\end{proof} 
	
	\begin{Corollary}\label{corr215}Let $R$ be a commutative ring such that $\textrm{sfli}R<\infty$ and $G$ be a group. Then ${\textrm{Gcd}}_{R}G<\infty$ if and only if there exists a characteristic module $\Lambda$ for $G$ over $R$.
	\end{Corollary}
	
	\begin{proof}This follows from Proposition \ref{newpropara}(ii) and \cite[Proposition 1.4]{ET}.
	\end{proof}

	\begin{Proposition}\label{newprop}Let $G$ be a group and $R$ be a commutative ring such that $\textrm{sfli}R<\infty$. Then ${\textrm{Gcd}}_{R}G<\infty$ if and only if ${\widetilde{\textrm{Gcd}}_{R}G}<\infty$. In that case, we have ${\textrm{Gcd}}_{R}G={\widetilde{\textrm{Gcd}}_{R}G}$.
	\end{Proposition}
	
	\begin{proof}The equivalence follows from Corollary \ref{cor38} and Corollary \ref{corr215}. Invoking Proposition \ref{newpropara}(ii), the finiteness of ${\textrm{Gcd}}_{R}G$ yields the existence of a characteristic module $\Lambda$ for $G$ over $R$ such that $\textrm{pd}_{RG}\Lambda = \textrm{Gcd}_{R}G$. Moreover, Corollary \ref{cor1} yields the existence of a characteristic module $\Lambda'$ for $G$ over $R$ such that $\textrm{pd}_{RG}\Lambda' = \widetilde{\textrm{Gcd}}_{R}G$. Then, Corollary \ref{cor37}(i) implies that $\textrm{pd}_{RG}\Lambda = \textrm{pd}_{RG}\Lambda'$, and hence we have $\textrm{Gcd}_{R}G=\widetilde{\textrm{Gcd}}_{R}G$.
	\end{proof}
	
	We now restate our first main theorem (Theorem 0.1) and deduce it from our previous results. 

\begin{Theorem}\label{theorr}
	Let $G$ be a group and $R$ be a commutative ring such that $\textrm{spli}R<\infty$. The following are equivalent:
	\begin{itemize}
		\item[(i)] ${\widetilde{\textrm{Gcd}}_{R}G}<\infty$.	
		\item[(ii)]${\textrm{Gcd}}_{R}G<\infty$.
		\item[(iii)] There exists a characteristic module $\Lambda$ for $G$ over $R$.
		\item[(iv)] Every $RG$-module has finite PGF dimension.
		\item[(v)] Every $RG$-module has finite Gorenstein projective dimension.
		\item[(vi)] $\textrm{silp}(RG)=\textrm{spli}(RG)< \infty$.
	\end{itemize}
	In this case, for every $RG$-module $M$ we have $\textrm{PGF-dim}_{RG}M=\textrm{Gpd}_{RG}M$.
\end{Theorem}

\begin{proof}
	Since $\textrm{spli}R<\infty$, by Proposition \ref{newprop} we have $(i)\Leftrightarrow(ii)$. Moreover, the equivalence $(i)\Leftrightarrow(iii)$ follows from Corollary \ref{cor1} and Corollary \ref{cor2}. The implication $(iv)\Rightarrow (v)$ is clear since the PGF dimension bounds the Gorenstein projective dimension, while the implication $(v)\Rightarrow (vi)$ is a consequence of \cite[Theorem 4.1]{Emm3}. Moreover, the equivalence $(iv)\Leftrightarrow(vi)$ follows from \cite[Theorem 5.1]{DE}, since $RG \cong {(RG)}^{\textrm{op}}$. Finally, invoking Proposition \ref{prop2} and \cite[Corollary 5.4]{DE}, we obtain the implication $(i)\Rightarrow (vi)$, while the implication $(iv)\Rightarrow (i)$ is trivial.
	
	If $\textrm{spli}(RG)<\infty$, then we also have $\textrm{sfli}(RG)<\infty$ and hence every syzygy of an complex of projective $RG$-modules is a PGF $RG$-module. Therefore, every Gorenstein projective $RG$-module is PGF. Invoking \cite[Theorem 4.4]{SS}, we infer that in this case the classes ${\tt GProj}(RG)$ and ${\tt PGF}(RG)$ are equal, and hence $\textrm{PGF-dim}_{RG}M=\textrm{Gpd}_{RG}M$ for every $RG$-module $M$.
\end{proof}

\subsection{Gorenstein global dimensions}The goal of this subsection is to give nice estimations for the invariant $\textrm{spli}(RG)$, in terms of the PGF dimension of the group $G$ and the invariant $\textrm{spli}R$. These estimations generalize also the upper bound for the invariants $\textrm{silp}(RG)$ and $\textrm{spli}(RG)$ given in \cite[Corollary 1.6]{ET}. By doing this, we provide estimations for the global dimension $\textrm{PGF-gl.dim}(RG)=\textrm{Ggl.dim}(RG)$. We provide also a subadditivity result for $\textrm{spli}(RG)$ under group extension, where the commutative ring $R$ is of finite Gorenstein global dimension.

Ikenaga defined in \cite{In} the generalized cohomological dimension $\underline{\textrm{cd}} G$ of a group $G$ over $\mathbb{Z}$ as $\underline{\textrm{cd}} G=\textrm{sup}\{i\in\mathbb{N}\,|\,\textrm{Ext}^i_{\mathbb{Z}G}(M,P)\neq 0,\, M \,\mathbb{Z}\textrm{-free},\, P \,\mathbb{Z}G\textrm{-projective}\}$. Here we naturally define the generalized cohomological dimension of a group $G$ over any commutative ring $R$ as $$\underline{\textrm{cd}}_R G=\textrm{sup}\{i\in\mathbb{N}\,|\,\textrm{Ext}^i_{RG}(M,P)\neq 0,\, M \,R\textrm{-projective},\, P \,RG\textrm{-projective}\}.$$ We note that $\underline{\textrm{cd}} G=\underline{\textrm{cd}}_{\mathbb{Z}} G$.

\begin{Proposition}\label{prop413}
	Let $R$ be a commutative ring and $G$ be a group such that ${\widetilde{\textrm{Gcd}}_{R}G}<\infty$. Then, ${\widetilde{\textrm{Gcd}}_{R}G}\leq \underline{\textrm{cd}}_R G$. Moreover, if $\textrm{spli}R<\infty$, then  ${\widetilde{\textrm{Gcd}}_{R}G}=\underline{\textrm{cd}}_R G$.
\end{Proposition}

\begin{proof}Since ${\widetilde{\textrm{Gcd}}_{R}G}=\textrm{PGF-dim}_{RG}R$ is finite, \cite[Proposition 3.6]{DE} implies that $${\widetilde{\textrm{Gcd}}_{R}G}=\textrm{sup}\{i\in\mathbb{N}\,|\,\textrm{Ext}^i_{RG}(R,P)\neq 0,\,P \,RG\textrm{-projective}\}.$$ Then, by the definition of $\underline{\textrm{cd}}_R G$ we have ${\widetilde{\textrm{Gcd}}_{R}G}\leq \underline{\textrm{cd}}_R G$. Moreover, if $\textrm{spli}R$ is finite, Theorem 2.23 implies that $\textrm{PGF-dim}_{RG}M<\infty$ for every $RG$-module $M$, and hence $\textrm{PGF-dim}_{RG}M =\textrm{sup}\{i\in\mathbb{N}\,|\,\textrm{Ext}^i_{RG}(M,P)\neq 0,\,P \,RG\textrm{-projective}\}$. Using again Theorem 2.23, we infer that there exists a characteristic module $\Lambda$ for $G$ which is such that ${\widetilde{\textrm{Gcd}}_{R}G}=\textrm{pd}_{RG}\Lambda$ (see Corollary 2.9). Then, Proposition 2.5 implies that $\textrm{PGF-dim}_{RG}M \leq \textrm{pd}_{RG}\Lambda = {\widetilde{\textrm{Gcd}}_{R}G}$ for every $R$-projective $RG$-module $M$. Consequently, we have $\underline{\textrm{cd}}_R G=\textrm{sup}\{\textrm{PGF-dim}_{RG}M\,|\, M \, R\textrm{-projective}\}$ $\leq {\widetilde{\textrm{Gcd}}_{R}G}$, and hence ${\widetilde{\textrm{Gcd}}_{R}G}=\underline{\textrm{cd}}_R G$.
\end{proof}

\begin{Lemma}\label{prop414}Let $R$ be a commutative ring and $G$ be a group. Then, $$\underline{\textrm{cd}}_R G\leq \textrm{silp}(RG) \leq \textrm{spli}(RG).$$\end{Lemma}

\begin{proof}For the first inequality, it suffices to assume that $\textrm{silp}(RG)=n<\infty$. Then, for every projective $RG$-module $P$ we have $\textrm{id}_{RG}P\leq n$ and hence $\textrm{Ext}^i_{RG}(M,P)=0$, for every $RG$-module $M$ and every $i>n$. It follows that $\underline{\textrm{cd}}_R G\leq n$, as needed. The second inequality follows from \cite[Corollary 5.4]{DE}, since $RG\cong {(RG)}^{\textrm{op}}$. \end{proof}

\begin{Proposition}\label{prop218}Let $R$ be a commutative ring and $G$ be a group. Then, $$\textrm{max}\{\textrm{spli}R, {\widetilde{\textrm{Gcd}}_{R}G}\}\leq \textrm{spli}(RG).$$
\end{Proposition}

\begin{proof}It suffices to assume that $\textrm{spli}(RG)=n$ is finite. We will show that $\textrm{spli}R\leq \textrm{spli}(RG)$ and ${\widetilde{\textrm{Gcd}}_{R}G}\leq \textrm{spli}(RG)$. The first inequality follows from Lemma \ref{lemmm22} for $H=\{1\}$. It remains to prove that ${\widetilde{\textrm{Gcd}}_{R}G}\leq \textrm{spli}(RG)$. Since $\textrm{spli}(RG)<\infty$ and $\textrm{spli}R<\infty$, invoking Theorem \ref{theorr}, we infer that ${\widetilde{\textrm{Gcd}}_{R}G}<\infty$. Therefore, Proposition \ref{prop413} yields ${\widetilde{\textrm{Gcd}}_{R}G}=\underline{\textrm{cd}}_R G$. Using Lemma \ref{prop414}, we conclude that ${\widetilde{\textrm{Gcd}}_{R}G}\leq \textrm{spli}(RG)$, as needed.\end{proof}

\begin{Corollary}\label{COR}Let $R$ be a commutative ring and $G$ be a group. Then, \begin{itemize}
	\item[(i)]$\textrm{max}\{\textrm{spli}R, {\widetilde{\textrm{Gcd}}_{R}G}\}\leq \textrm{spli}(RG)\leq {\widetilde{\textrm{Gcd}}_{R}G} + \textrm{spli}R$,
	\item[(ii)]$\textrm{max}\{\textrm{spli}R, {\textrm{Gcd}}_{R}G\}\leq \textrm{spli}(RG)\leq {\textrm{Gcd}}_{R}G + \textrm{spli}R.$
\end{itemize}
\end{Corollary}

\begin{proof}(i) This follows immediately from Proposition \ref{prop2} and Proposition \ref{prop218}.
	
	(ii) It suffices to assume that ${\textrm{Gcd}}_{R}G<\infty$ and $\textrm{spli}R<\infty$. Then, the inequalities follow from (i) above and Proposition \ref{newprop}.
\end{proof}

\begin{Corollary}\label{corrr}
	Let $R$ be a commutative ring and $G$ be a group. Then, $$ \textrm{max}\{\textrm{PGF-gl.dim}R, {\widetilde{\textrm{Gcd}}_{R}G}\}\leq\textrm{PGF-gl.dim}(RG) \leq {\widetilde{\textrm{Gcd}}_{R}G} + \textrm{PGF-gl.dim}R.$$
\end{Corollary}

\begin{proof}It suffices to assume that $\textrm{PGF-gl.dim}R<\infty$ and $\widetilde{\textrm{Gcd}}_{R}G<\infty$. Invoking \cite[Theorem 5.1]{DE}, we infer that $\textrm{spli}R=\textrm{PGF-gl.dim}R<\infty$. Moreover, Corollary \ref{COR}(i) yields the finiteness of $\textrm{spli}(RG)$. Since $RG\cong {(RG)^{\textrm{op}}}$ and $\textrm{sfli}(RG)\leq \textrm{spli}(RG)<\infty$, using again \cite[Theorem 5.1]{DE}, we obtain that $\textrm{PGF-gl.dim}(RG)=\textrm{spli}(RG)<\infty$. The inequalities now follow from Corollary \ref{COR}(i).
\end{proof}

\begin{Corollary}Let $R$ be a commutative ring and $G$ be a group. Then, $$\textrm{max}\{\textrm{Ggl.dim}R, \textrm{Gcd}_{R}G\} \leq\textrm{Ggl.dim}(RG) \leq {{\textrm{Gcd}}_{R}G} + \textrm{Ggl.dim}R.$$
\end{Corollary}

\begin{proof}It suffices to assume that ${\textrm{Gcd}}_{R}G$ and $\textrm{Ggl.dim}R$ are finite. Then, \cite[Theorem 4.1]{Emm3} implies that $\textrm{spli}R<\infty$ and hence ${\textrm{Gcd}}_{R}G=\widetilde{\textrm{Gcd}}_{R}G<\infty$, by Proposition \ref{newprop}.  Furthermore, we have $\textrm{PGF-gl.dim}R=\textrm{Ggl.dim}R$ and $\textrm{PGF-gl.dim}(RG)=\textrm{Ggl.dim}(RG)$ by Lemma \ref{lemgl}. Therefore, the inequalities follow from Corollary \ref{corrr}.
\end{proof}

The following proposition gives a generalization of \cite[Theorem 5.5]{In} which is stated over $\mathbb{Z}$, over any commutative ring of finite Gorenstein global dimension.

\begin{Proposition}\label{prop65}Let $R$ be a commutative ring such that $\textrm{spli}R<\infty$ and consider a group $G$, a normal subgroup $H$ of $G$ and the corresponding quotient group $Q=G/H$. Then, $$\textrm{spli}(RG) \leq \textrm{spli}(RH) + \widetilde{\textrm{Gcd}}_{R}Q \leq \textrm{spli}(RH) + \textrm{spli}(RQ).$$
\end{Proposition}

\begin{proof}First we note that the quotient homomorphism $G\rightarrow Q$ enables us to regard every $RQ$-module as an $RG$-module and every $RQ$-linear map as an $RG$-linear map. It suffices to assume that $\textrm{spli}(RH)=n$ and $\textrm{spli}(RQ)=m$ are finite. Since $\textrm{spli}R<\infty$, Theorem \ref{theorr} yields the existence of an $R$-split monomorphism of $RQ$-modules $\iota: R \rightarrow \Lambda$, where $\Lambda$ is an $R$-projective $RQ$-module and $\textrm{pd}_{RQ}A =m'<\infty$. We note that $m'=\widetilde{\textrm{Gcd}}_{R}Q\leq m$, by Corollary \ref{cor1} and Lemma \ref{prop414}. It follows that there exists an exact sequence of $RQ$-modules:
	\begin{equation*}\label{eq1}\textbf{Q}= 0\rightarrow P_{m'} \rightarrow \cdots \rightarrow P_1\rightarrow P_0 \rightarrow \Lambda \rightarrow 0,
	\end{equation*}
	where the $RQ$-module $P_i$ is projective for every $i=0,1,\dots,m'$. We consider now an injective $RG$-module $I$ and a truncated $RG$-projective resolution of $I$ of length $n$:
	\begin{equation*}\label{eq5}
		\textbf{P}=	0\rightarrow I_n \rightarrow P'_{n-1} \rightarrow \cdots \rightarrow P'_0 \rightarrow I \rightarrow 0.
	\end{equation*}
	Since $\textrm{spli}(RH)=n$, we obtain that the $RH$-modules $I_n,P'_{n-1},\dots ,P'_0$ are projective. Since $\Lambda$ is $R$-projective (and hence $R$-flat), it follows from Künneth's formula that the total complex of the double complex $\textbf{Q}\otimes_R \textbf{P}$ yields a resolution of $\Lambda\otimes_R I$ by $RG$-modules of length $n+m'$. Invoking Lemma \ref{lemZ}, we infer that this resolution consists of projective $RG$-modules and hence $\textrm{pd}_{RG}(\Lambda\otimes_R I )\leq n+m'$. The $R$-split monomorphism of $RG$-modules $\iota: R \rightarrow \Lambda$ yields an $RG$-monomorphism $I=R\otimes_R I \rightarrow \Lambda\otimes_R I$, which is $RG$-split, and hence $\textrm{pd}_{RG}I\leq \textrm{pd}_{RG}(\Lambda\otimes_R I )\leq n+m'\leq n+m$. It follows that $\textrm{spli}(RG)\leq n+m'\leq n+m$, as needed.\end{proof}

\begin{Corollary}\label{cor2223}Let $R$ be a commutative ring such that $\textrm{spli}R<\infty$ and consider a group $G$, a normal subgroup $H$ of $G$ and the corresponding quotient group $G/H$. Then, $$\textrm{PGF-gl.dim}(RG) \leq  \textrm{PGF-gl.dim}(RH) + \textrm{PGF-gl.dim}(RQ).$$
\end{Corollary}

\begin{proof}It suffices to assume that $\textrm{PGF-gl.dim}(RH)<\infty$ and $\textrm{PGF-gl.dim}(RQ)<\infty$. Then, invoking \cite[Theorem 5.1]{DE}, we infer $\textrm{PGF-gl.dim}(RH)=\textrm{spli}(RH)<\infty$ and $\textrm{PGF-gl.dim}(RQ)=\textrm{spli}(RQ)<\infty$. Therefore, Proposition \ref{prop65} implies that $\textrm{spli}(RG) \leq \textrm{spli}(RH) + \textrm{spli}(RQ)<\infty$. Since $RG\cong {(RG)}^{\textrm{op}}$ and $\textrm{sfli}(RG)\leq \textrm{spli}(RG)<\infty$, invoking again \cite[Theorem 5.1]{DE}, we obtain that $\textrm{PGF-gl.dim}(RG)=\textrm{spli}(RG)<\infty$ and hence $\textrm{PGF-gl.dim}(RG) \leq  \textrm{PGF-gl.dim}(RH) + \textrm{PGF-gl.dim}(RQ)$, as needed.
\end{proof}

\begin{Corollary}Let $R$ be a commutative ring such that $\textrm{spli}R<\infty$ and consider a group $G$, a normal subgroup $H$ of $G$ and the corresponding quotient group $G/H$. Then, $$\textrm{Ggl.dim}(RG) \leq \textrm{Gwgl.dim}(RH) + \textrm{Gwgl.dim}(RQ).$$
\end{Corollary}

\begin{proof}Since $RG\cong {(RG)}^{\textrm{op}}$, $RH\cong {(RH)}^{\textrm{op}}$ and $RQ\cong {(RQ)}^{\textrm{op}}$ this is a direct consequence of Corollary \ref{cor2223} and Lemma \ref{lemgl}.
\end{proof}

\subsection{A bound for the PGF dimension over group rings} 
Our goal in this subsection is to construct an upper bound for the PGF dimension $\textrm{PGF-dim}_{RG}M$ of a $RG$-module $M$, in terms of PGF dimension $\widetilde{\textrm{Gcd}}_{R}G$ of $G$ over $R$ and the PGF dimension $\textrm{PGF-dim}_{R}M$ of the restricted $k$-module $M$. By doing this, we also provide a PGF analogue of the well-known bound of the projective dimension $\textrm{pd}_{RG}M$ of an $RG$-module $M$, in terms of the cohomological dimension $\textrm{cd}_R G$ of the group $G$ and the projective dimension $\textrm{pd}_R M$ of the restricted $R$-module $M$.

\begin{Lemma}\label{newlemmad}Let $R$ be a commutative ring and $G$ be a group. 
	\begin{itemize}
		\item[(i)] For every $RG$-modules $M$, $N$ such that $M$ is $R$-projective and $N$ is PGF as $R$-module, we have $\textrm{PGF-dim}_{RG}(M\otimes_{R}N)\leq \textrm{pd}_{RG}M$.
		\item[(ii)] Let $\textrm{spli}R<\infty$ and $\Lambda$ be a characteristic module for $G$ over $R$. Then for every $RG$-modules $M$, $N$ such that $M$ is PGF and $N$ is R-projective, the $RG$-module $M\otimes_R N$ is PGF.
	\end{itemize}
\end{Lemma}

\begin{proof}(i) It suffices to assume that $\textrm{pd}_{RG}M=n<\infty$. Consider an $RG$-projective resolution $0 \rightarrow P_n \rightarrow \cdots \rightarrow P_1\rightarrow P_0 \rightarrow M \rightarrow 0$ of $M$. Since $M$ is $R$-projective, the projective resolution above is $R$-split. Thus we obtain an induced exact sequence of $RG$-modules with diagonal action $0\rightarrow P_n\otimes_{R}N \rightarrow \cdots \rightarrow P_1\otimes_{R}N \rightarrow P_0\otimes_{R}N \rightarrow M\otimes_{R}N \rightarrow 0$ which constitutes a PGF resolution of the $RG$-module $M\otimes_{R}N$ by Lemma \ref{lem46} (ii). We conclude that $\textrm{PGF-dim}_{RG}(M\otimes_{R}N)\leq \textrm{pd}_{RG}M$, as needed.
	
	(ii) Let $M$ be a PGF $RG$-module and $N$ be an $R$-projective $RG$-module. Then, there exists an acyclic complex of projective $RG$-modules $$\textbf{P}=\cdots \rightarrow P_{2}\rightarrow P_1\rightarrow P_0 \rightarrow P_{-1}\rightarrow \cdots,$$ such that $M=\textrm{Im}(P_1 \rightarrow P_0)$ and the complex $I\otimes_{RG}\textbf{P}$ is exact, whenever $I$ is an injective $RG$-module. Since $N$ is $R$-projective, we obtain the induced complex of $RG$-projective modules $$\textbf{P}\otimes_R N = \cdots \rightarrow P_{2}\otimes_R N\rightarrow P_1\otimes_R N\rightarrow P_0\otimes_R N \rightarrow P_{-1}\otimes_R N\rightarrow \cdots,$$ where $M\otimes_R N= \textrm{Im}(P_1\otimes_R N \rightarrow P_0\otimes_R N)$. The existence of the characteristic module $\Lambda$ implies that ${\widetilde{\textrm{Gcd}}_{R}G}\leq \textrm{pd}_{RG} \Lambda<\infty$ by Corollary \ref{cor2}. Thus, Proposition \ref{prop2} yields $\textrm{sfli}(RG)\leq\textrm{spli}(RG)<\infty$ and hence the complex $I\otimes_{RG}(\textbf{P}\otimes_R N)$ is acyclic for every injective $RG$-module $I$. We conclude that the $RG$-module $M\otimes_R N$ is PGF, as needed. 
\end{proof}

Since every projective module is PGF, the next result is a generalization of Proposition \ref{prop1}. 

\begin{Proposition}\label{popara}Let $G$ be a group and $R$ be a commutative ring such that $\textrm{spli}R<\infty$. Consider a characteristic module $\Lambda$ for $G$ over $R$. Then, for every $RG$-module $M$ which is PGF as $R$-module, we have $\textrm{PGF-dim}_{RG}M \leq \textrm{pd}_{RG} \Lambda$.
\end{Proposition}

\begin{proof} Let $\textrm{pd}_{RG} \Lambda =n$ and $M$ be an $R$-projective $RG$-module. We consider an $RG$-projective resolution $\textbf{P}=\cdots \rightarrow P_2 \rightarrow P_1 \rightarrow P_0 \rightarrow M \rightarrow 0$ of $M$ and set $M_i=\textrm{Im}(P_i \rightarrow P_{i-1})$, $i\geq 0$, where $M_0=M$, the corresponding syzygy modules. It suffices to prove that the $RG$-module $M_n$ is PGF. Since the characteristic module $\Lambda$ is $R$-projective, $\textbf{P} \otimes_R \Lambda$ is an $RG$-projective resolution of $M\otimes_R \Lambda$ with $i$-th syzygies the modules $M_i \otimes_R \Lambda$, $i\geq 0$. Using Lemma \ref{newlemmad}(i) we obtain that $\textrm{PGF-dim}_{RG}(M\otimes_R \Lambda)\leq \textrm{pd}_{RG}\Lambda=n$. Thus, the $RG$-module $M_i \otimes_R \Lambda$ is PGF for every $i\geq n$. Moreover, Lemma \ref{newlemmad}(ii) implies that the $RG$-module $M_n \otimes_R N \otimes_R \Lambda$ is also PGF for every $R$-projective $RG$-module $N$. Let $L=\textrm{Coker}\iota$ and consider the $R$-split short exact sequence of $RG$-modules $0\rightarrow R \xrightarrow{\iota} \Lambda \rightarrow L \rightarrow 0$. Since $L$ is $R$-projective, following the proof of Proposition \ref{prop1} we obtain the acyclic complex of PGF $RG$-modules
	$$\mathfrak{P}=\cdots \rightarrow P_{n+2}\rightarrow P_{n+1}\rightarrow P_n \xrightarrow{\eta \epsilon}  M_n \otimes_R \Lambda \rightarrow M_n \otimes_R L\otimes_R \Lambda \rightarrow M_n \otimes_R L^{\otimes 2}\otimes_R \Lambda \rightarrow \cdots ,$$ which has syzygies the $RG$-modules $(M_i)_{i\geq n}$ and $(M_n \otimes_R L^{\otimes j})_{j\geq 1}$. 
	In order to prove that the $RG$-module $M_n$ is PGF, using the stability result \cite[Theorem 6.7]{St}, it suffices to show that the complex $ I\otimes_{RG} \mathfrak{P}$ is acyclic for every injective $RG$-module $I$. Let $I$ be an injective $RG$-module. Then, as in the proof of Proposition \ref{prop1}, it suffices to show the acyclicity of the complex $(\Lambda\otimes_R I)\otimes_{RG}\mathfrak{P}$. Since the $RG$-modules $\Lambda$ and $L$ are $R$-projective, the complex of $RG$-modules $\mathfrak{P}\otimes_R \Lambda$ is acyclic with syzygies the PGF $RG$-modules $(M_i \otimes_R \Lambda)_{i\geq n}$ and $(M_n \otimes_R L^{\otimes j} \otimes_R \Lambda)_{j\geq 1}$. Moreover, for every PGF RG-module $K$ we have $\textrm{Tor}_1^{RG}(I,K)=0$ and hence the complex $I\otimes_{RG}(\mathfrak{P}\otimes_R \Lambda)\cong (\Lambda\otimes_R I)\otimes_{RG}\mathfrak{P}$ is acyclic. 
\end{proof}
\begin{Corollary}\label{cooor216} Let $G$ be a group and $R$ be a commutative ring such that $\textrm{spli}R<\infty$. Then, $$\textrm{PGF-dim}_{RG}M\leq {\widetilde{\textrm{Gcd}}_{R}G}+\textrm{PGF-dim}_{R}M.$$
\end{Corollary}
\begin{proof}
	It suffices to assume that ${\widetilde{\textrm{Gcd}}_{R}G}=n<\infty$ and $\textrm{PGF-dim}_{R}M=m<\infty$. Then, Corollary \ref{cor1} yields the existence of a characteristic module $\Lambda$ for $G$ over $R$ such that $\textrm{pd}_{RG} \Lambda=n$. Consider an $RG$-projective resolution $\textbf{P}= \cdots  \rightarrow P_1 \rightarrow P_0 \rightarrow M \rightarrow 0 $ of $M$ and let $K_m$ be the corresponding $m$-th syzygy. Since $\textrm{PGF-dim}_{R}M=m$, we obtain that $K_m$ is PGF as $R$-module. Thus, Proposition \ref{popara} implies that $\textrm{PGF-dim}_{RG}K_m \leq n$. Using Lemma \ref{lem63} and the exact sequence $$0\rightarrow K_m \rightarrow P_{m-1} \rightarrow  \cdots \rightarrow P_1 \rightarrow P_0 \rightarrow M \rightarrow 0$$ we conclude that $\textrm{PGF-dim}_{RG}M\leq m+n$.
\end{proof}

\begin{Remark}\rm We note that the right inequality in Corollary \ref{corrr} can be also obtained from Corollary \ref{cooor216}.
\end{Remark}

\section{Properties of the PGF dimension of groups}
The goal of this section is to establish analogous properties for the PGF dimension of groups to those enjoyed by the cohomological dimension and the Gorenstein cohomological dimension (see \cite{ET}). We examine first the dependence of the PGF dimension of a group $G$ upon the coefficient ring $R$. We also examine the relation between the PGF dimension of a group and and its subgroups. Our main tool is the existence of a characteristic module obtain from Corollary \ref{cor1}. Finally, we obtain Gorenstein analogues for Serre's theorem.

\begin{Proposition}Let $G$ be a group and $R,S$ be commutative rings such that $\textrm{sfli}R<\infty$ and $S$ is an extension of $R$. Then, ${\widetilde{\textrm{Gcd}}_{S}G}\leq {\widetilde{\textrm{Gcd}}_{{R}}G}$.
\end{Proposition}

\begin{proof}It suffices to assume that ${\widetilde{\textrm{Gcd}}_{{R}}G}=n$ is finite. Then, by Corollary \ref{cor1} there exists a $R$-split monomorphism of ${R}G$-modules $\iota: R\rightarrow \Lambda$, where $\Lambda$ is an $R$-projective $RG$-module and $\textrm{pd}_{RG}\Lambda =n$. Then, we have the $S$-split monomorphism of $SG$-modules $S\rightarrow S\otimes_R\Lambda$, where the $SG$-module $S\otimes_R\Lambda$ is $S$-projective. Since $\textrm{pd}_{RG}\Lambda =n$, there exists an exact sequence of $RG$-modules $$0\rightarrow P_n \rightarrow \cdots \rightarrow P_1 \rightarrow P_0 \rightarrow \Lambda \rightarrow 0,$$ where the $RG$-modules $P_i$ are projective for every $i=0,\dots ,n$. Since $\Lambda$ is $R$-projective, the exact sequence above is $R$-split. Thus, we obtain the exact sequence of $SG$-modules $$0\rightarrow S\otimes_{R}P_n \rightarrow \cdots \rightarrow S\otimes_{R}P_1\rightarrow S\otimes_{R}P_0\rightarrow S\otimes_{R}\Lambda\rightarrow 0,$$ where the $SG$-modules $S\otimes_{R}P_i$ are projective for every $i=0,\dots ,n$. Therefore, $\textrm{pd}_{SG}(S\otimes_{R}\Lambda)\leq n$. Using Corollary \ref{cor2} we conclude that ${\widetilde{\textrm{Gcd}}_{S}G}\leq \textrm{pd}_{SG}(S\otimes_{R}\Lambda)\leq n$.
\end{proof}

\begin{Corollary}\label{cor46}
Let $G$ be a group. Then, ${\widetilde{\textrm{Gcd}}_{R}G}\leq {\widetilde{\textrm{Gcd}}_{\mathbb{Z}}G}$ for every commutative ring $R$.
\end{Corollary}

%\begin{Corollary} Let $R$ be a commutative ring and $G$ be a group. Then, for every $RG$-module $M$ with finite PGF dimension, we have $\textrm{PGF-dim}_{RG}M \leq {\widetilde{\textrm{Gcd}}_{\mathbb{Z}}G} +\textrm{pd}_{R}M$. \end{Corollary}\begin{proof} This is an immediate consequence of Proposition \ref{prop64} and Corollary \ref{cor46}.\end{proof}

\begin{Theorem}\label{theo1}
	Let $G$ be a group. Then, the following are equivalent:
	\begin{itemize}
		\item[(i)] G is a finite group.
		\item[(ii)] ${\widetilde{\textrm{Gcd}}_{R}G}=0$ for every commutative ring $R$.
		\item[(iii)] ${\widetilde{\textrm{Gcd}}_{\mathbb{Z}}G}=0$.
	\end{itemize}
\end{Theorem}

\begin{proof} $(i)\Rightarrow (iii):$ Let $G$ be a finite group. Then, by \cite[VI Proposition 2.6]{Br}, there exists an acyclic complex of projective $\mathbb{Z}G$-modules $\textbf{P}=\cdots \rightarrow P_1 \rightarrow P_0 \rightarrow P_{-1}\rightarrow \cdots$, which has $\mathbb{Z}$ as syzygy. Let $K_n=\textrm{Im}(P_n \rightarrow P_{n-1})$ and $H={1}$. Then ${K_n}|_H$ is $\mathbb{Z}$-free and the Eckman-Shapiro lemma yields $\textrm{Tor}^{\mathbb{Z}G}_1(I,K_n)=0$ for every injective $\mathbb{Z}G$-module $I$. Thus, the complex $I \otimes_{\mathbb{Z}G} \textbf{P} $ is acyclic for every injective $\mathbb{Z}G$-module $I$ and the $\mathbb{Z}G$-module $\mathbb{Z}$ is $PGF$.
	
	$(iii)\Rightarrow (i):$ Since $\textrm{Gcd}_{\mathbb{Z}}G\leq {\widetilde{\textrm{Gcd}}_{\mathbb{Z}}G}$, we have $\textrm{Gcd}_{\mathbb{Z}}G=0$ and \cite[Corollary 2.3]{ET} implies that the group $G$ is finite.
	
	$(ii)\Leftrightarrow (iii):$ This is an immediate consequence of Corollary \ref{cor46}.\end{proof}

\begin{Corollary}\label{cor44}Let $R$ be a commutative ring and $G$ be a finite group. Then, $\textrm{spli}(RG)=\textrm{spli}R$.
\end{Corollary}

\begin{proof}This follows from Corollary \ref{COR}(i) and Theorem \ref{theo1}.
\end{proof}

We denote by $\textrm{Ghd}_R G$ the Gorenstein flat dimension of the trivial $RG$-module $R$. A consequence of Theorem \ref{theo1} is the following result.

\begin{Corollary}\label{corr35}Let $R$ be a commutative ring and $G$ be a finite group. Then, $\textrm{Ghd}_R G=0$. 
\end{Corollary}

\begin{proof}This is an immediate consequence of Theorem \ref{theo1}, since the PGF dimension bounds the Gorenstein flat dimension of every $RG$-module.
\end{proof}

\begin{Remark}\rm The inverse of Corollary \ref{corr35} cannot be true, since for every infinite locally finite group we have $\textrm{Ghd}_{R} G=0$. Indeed, let $G$ be a locally finite group. Then, $G={\lim\limits_{\longrightarrow}}_{i}G_i$ where $G_i$ are the finite subgroups of $G$. Since for every finite group $G_i$ the $RG_i$ module $R$ is PGF (see Theorem \ref{theo1}), invoking the isomorphism $R\cong {\lim\limits_{\longrightarrow}}_{i}\textrm{Ind}^G_{G_i}R$ and Lemma \ref{lem46}(i), we obtain that the $RG$-module $R$ is a direct limit of PGF modules and hence a direct limit of Gorenstein flat modules. Invoking \cite[Corollary 4.12]{SS}, we conclude that the $RG$-module $R$ is Gorenstein flat and hence $\textrm{Ghd}_{R}G=0$.
\end{Remark}

\begin{Corollary}\label{cor52} Let $R$ be a commutative ring and consider a group $G$ and a normal subgroup $H$ of $G$. Then, for every projective $R[G/H]$-module $P$ we have $\textrm{PGF-dim}_{RG}P \leq {\widetilde{\textrm{Gcd}}_R H}$, where $P$ is viewed as an $RG$-module via the quotient homomorphism $G\rightarrow G/H$. 
\end{Corollary}

\begin{proof}Since the PGF dimension of the trivial $RH$-module $R$ is equal to ${\widetilde{\textrm{Gcd}}_R H}$ and $\textrm{Ind}^G_H R=R[G/H]$, Lemma \ref{lem46}(iii) yields $\textrm{PGF-dim}_{RG} (R[G/H]) \leq {\widetilde{\textrm{Gcd}}_R H}$. Then, \cite[Proposition 2.3]{DE} implies that for every free $R[G/H]$-module $F$ we also have $\textrm{PGF-dim}_{RG}F\leq {\widetilde{\textrm{Gcd}}_R H}$. Applying again \cite[Proposition 2.3]{DE}, we obtain that $\textrm{PGF-dim}_{RG}P \leq {\widetilde{\textrm{Gcd}}_R H}$ for every projective $R[G/H]$-module. \end{proof}

\begin{Proposition}\label{prop38}Let $G$ be a group and $H$ be a subgroup of $G$. Then, for every commutative ring such that $\textrm{sfli}R<\infty$ we have ${\widetilde{\textrm{Gcd}}_R H}\leq {\widetilde{\textrm{Gcd}}_{R}G}$.
\end{Proposition}

\begin{proof}It suffices to assume that ${\widetilde{\textrm{Gcd}}_{R}G}=n$ is finite. It follows from Corollary \ref{cor1} that there exists an $R$-split monomorphism of $RG$-modules $\iota: R \rightarrow \Lambda$, where $\Lambda$ is an $R$-projective $RG$-module and $\textrm{pd}_{RG}\Lambda =n$. We note that the restriction of every projective $RG$-module to the subgroup $H$ is projective $RH$-module and that the $R$-split monomorphism $\iota$ is restricted to an $R$-split monomorphism of $RH$-modules $\iota|_H: R \rightarrow \Lambda$. Since $\textrm{pd}_{RH}{\Lambda}\leq \textrm{pd}_{RG}{\Lambda}= n$, Corollary \ref{cor2} yields ${\widetilde{\textrm{Gcd}}_R H}\leq n$.
\end{proof}

%\begin{Corollary}Let $G$ be a group and $H$ be a subgroup of $G$. Then, ${\widetilde{\textrm{Gcd}}_{\mathbb{Z}} H}\leq {\widetilde{\textrm{Gcd}}_{\mathbb{Z}}G}$ and ${\widetilde{\textrm{Gcd}}_{\mathbb{Q}} H}\leq {\widetilde{\textrm{Gcd}}_{\mathbb{Q}}G}$.\end{Corollary}

\begin{Proposition}\label{prop56} Let $R$ be a commutative ring such that $\textrm{sfli}R<\infty$ and consider a group $G$ and a normal subgroup $H$ of $G$. Then, ${\widetilde{\textrm{Gcd}}_{R} G}\leq {\widetilde{\textrm{Gcd}}_{R}}H + {\widetilde{\textrm{Gcd}}_{R}}(G/H)$.
\end{Proposition}

\begin{proof}Let $\overline{G}=G/H$ and note that the quotient homomorphism $G\rightarrow \overline{G}$ enables us to regard every $R\overline{G}$-module as an $RG$-module and every $R\overline{G}$-linear map as an $RG$-linear map. It suffices to assume that ${\widetilde{\textrm{Gcd}}_{R}}H=n$ and ${\widetilde{\textrm{Gcd}}_{R}}\overline{G}=m$ are finite. Then, by Corollary \ref{cor1} there exists an $R$-split monomorphism of $R\overline{G}$-modules $\iota: R \rightarrow \Lambda$, where $\Lambda$ is an $R$-projective $R\overline{G}$-module and $\textrm{pd}_{R\overline{G}}\Lambda =m$. It follows that there exists an exact sequence of $R\overline{G}$-modules $$0\rightarrow P_m \rightarrow \cdots \rightarrow P_1\rightarrow P_0 \rightarrow \Lambda \rightarrow 0,$$ where the $R\overline{G}$-module $P_i$ is projective for every $i=0,1,\dots,m$. Then, using Corollary \ref{cor52} we have $\textrm{PGF-dim}_{RG}P_i\leq n$ for every $i=0,1,\dots,m$ and Lemma \ref{lem63} implies that $\textrm{PGF-dim}_{RG}\Lambda \leq n+m$. Thus, by Proposition \ref{Prop} we obtain that there exists an $R$-split $RG$-monomorphism $j: \Lambda \rightarrow N$, where $N$ is $R$-projective and $\textrm{pd}_{RG}N=\textrm{PGF-dim}_{RG}\Lambda\leq n+m$. We consider now the composition $R\xrightarrow{\iota} \Lambda \xrightarrow{j}N$, which is also an $R$-split monomorphism of $RG$-modules. Then, using Corollary \ref{cor2} we conclude that ${\widetilde{\textrm{Gcd}}_{R} G}\leq \textrm{pd}_{RG}N\leq n+m$.
\end{proof}

\begin{Corollary}\label{cor313}Let $R$ be a commutative ring such that $\textrm{sfli}R<\infty$ and consider a group $G$ and a normal subgroup $H$ of $G$. If the quotient group $G/H$ is finite, then ${\widetilde{\textrm{Gcd}}_{R} G} ={\widetilde{\textrm{Gcd}}_{R}}H$.
\end{Corollary}

\begin{proof}
Since the quotient group $G/H$ is finite, Theorem \ref{theo1} yields ${\widetilde{\textrm{Gcd}}_{R} (G/H)}=0$. Then, Proposition \ref{prop56} implies that ${\widetilde{\textrm{Gcd}}_{R} G}\leq {\widetilde{\textrm{Gcd}}_{R}}H$. Moreover, by Proposition \ref{prop38} we have ${\widetilde{\textrm{Gcd}}_{R} H}\leq {\widetilde{\textrm{Gcd}}_{R}}G$. We conclude that ${\widetilde{\textrm{Gcd}}_{R} G} ={\widetilde{\textrm{Gcd}}_{R}}H$.
\end{proof}

\begin{Remark}\label{rem510}\rm Let $G$ be a group and $H$ be a finite normal subgroup of $G$. Then, the modules over the group $G/H$ are precisely the modules over the group $G$ on which the subgroup $H$ acts trivially, i.e. $\textrm{R[G/H]-Mod}=\{M\in \textrm{RG-Mod}:M^H=M\}$. Moreover, the finiteness of $H$ yields the following consequences:

	(i) For every projective $RG$-module, the $R[G/H]$-module $P^H$ is projective as well. 
	
	Indeed, it suffices to consider the case where $P=RG$ and notice that $(RG)^H\cong R[G/H]$ as $R[G/H]$-modules.
	
(ii) For every $R$-projective $RG$-module $M$ such that $\textrm{pd}_{RG}<\infty$, the $R[G/H]$-module $M^H$ is also $R$-projective and $\textrm{pd}_{R[G/H]}M^H\leq \textrm{pd}_{RG}M$. 

Indeed, let $M$ be an $R$-projective $RG$-module such that $\textrm{pd}_{RG}=n<\infty$. Then, there exists an exact sequence of $RG$-modules 
		\begin{equation}\label{eq3}
		0\rightarrow P_n \rightarrow \cdots \rightarrow P_1 \rightarrow P_0 \rightarrow M \rightarrow 0,
		\end{equation} where $P_i$ is projective $RG$-module for every $i=0,1,\dots,n$. Since $M$ is $R$-projective, the exact sequence (\ref{eq3}) is $R$-split. Moreover, we regard the exact sequence (\ref{eq3}) as an exact sequence of $RH$-modules, where the $RH$-module $P_i$ is projective for every $i=0,1,\dots ,n$. Since the subgroup $H$ is finite, \cite[VI Corollary 2.3]{Br} implies that the exact sequence (\ref{eq3}) is $RH$-split and hence the $RH$-module $M$ is projective. Then, using the isomorphism $(RH)^H\cong R$, we conclude that $M^H$ is $R$-projective. Moreover, as the exact sequence (\ref{eq3}) is $RH$-split, we obtain an exact sequence of $R[G/H]$-modules $$0\rightarrow P_n^H \rightarrow \cdots \rightarrow P_1^H \rightarrow P_0^H\rightarrow M^H\rightarrow 0,$$ where the $R[G/H]$-module $P_i^H$ is projective for every $i=0,1,\dots ,n$ by (i).
\end{Remark}

\begin{Proposition}\label{prop520}Let $R$ be a commutative ring such that $\textrm{sfli}R<\infty$, $G$ be a group and $H$ be a finite normal subgroup of $G$. Then, ${\widetilde{\textrm{Gcd}}_{R}G}={\widetilde{\textrm{Gcd}}_{R}(G/H)}$.
\end{Proposition}

\begin{proof}Proposition \ref{prop56} yields ${\widetilde{\textrm{Gcd}}_{R} G}\leq {\widetilde{\textrm{Gcd}}_{R}}H + {\widetilde{\textrm{Gcd}}_{R}}(G/H)$. Since the subgroup $H$ is finite, Theorem \ref{theo1} implies that ${\widetilde{\textrm{Gcd}}_{R}}H=0$ and hence we have ${\widetilde{\textrm{Gcd}}_{R} G}\leq {\widetilde{\textrm{Gcd}}_{R}}(G/H)$. It remains to prove the inequality ${\widetilde{\textrm{Gcd}}_{R} (G/ H)}\leq {\widetilde{\textrm{Gcd}}_{R}}G$. For that, it suffices to assume that ${\widetilde{\textrm{Gcd}}_{R}}G=n<\infty$. Then, Corollary \ref{cor1} implies that there exists an $R$-split monomorphism of $RG$-modules $\iota: R \rightarrow \Lambda$, where $\Lambda$ is an $R$-projective $RG$-module and $\textrm{pd}_{RG}\Lambda =n$. Since the group $G$ acts trivially on $R$, we have $\textrm{Im}\iota \subseteq \Lambda^G \subseteq \Lambda^H$. We may therefore consider the $R[G/H]$-module $\Lambda^H$ and the $R$-split $R[G/H]$-monomorphism $j:R\rightarrow \Lambda^H$. Invoking Remark \ref{rem510}(ii) we obtain that the $R[G/H]$-module $\Lambda^H$ is $R$-projective and $\textrm{pd}_{R[G/H]}\Lambda^H\leq \textrm{pd}_{RG}\Lambda =n$. Then, Corollary \ref{cor2} yields ${\widetilde{\textrm{Gcd}}_{R}}(G/H)\leq \textrm{pd}_{R[G/H]}\Lambda^H\leq n$.
\end{proof}

\begin{Corollary}\label{cor216}Let $R$ be a commutative ring such that $\textrm{sfli}R<\infty$, $G$ be a group and $H$ be a finite subgroup of $G$. Then, for the Weyl group $W=N_G(H)/H$ we have ${\widetilde{\textrm{Gcd}}_{R}W}\leq{\widetilde{\textrm{Gcd}}_{R}G}$.
\end{Corollary}

\begin{proof}Invoking Proposition \ref{prop38} we have ${\widetilde{\textrm{Gcd}}_{R}}N_G(H)\leq {\widetilde{\textrm{Gcd}}_{R}(G)}$. Moreover, Proposition \ref{prop520} yields ${\widetilde{\textrm{Gcd}}_{R}(N_G(H))}={\widetilde{\textrm{Gcd}}_{R}W}$. We conclude that ${\widetilde{\textrm{Gcd}}_{R}W}\leq{\widetilde{\textrm{Gcd}}_{R}(G)}$.\end{proof}

 Serre's Theorem \cite[VIII Theorem 3.1]{Br} yields an equality between cohomological dimensions of a group and subgroups with finite index. In the next result we give an analogous result concerning PGF and Gorenstein cohomological dimensions of groups.

\begin{Theorem}Let $R$ be a commutative ring such that $\textrm{spli}R<\infty$ and consider a group $G$ and a subgroup $H$ of $G$ of finite index. Then, ${\widetilde{\textrm{Gcd}}_{R}G}={\widetilde{\textrm{Gcd}}_{R}H}$ and $\textrm{Gcd}_{R}G=\textrm{Gcd}_{R}H$.
\end{Theorem}

\begin{proof}Since the subgroup $H$ of $G$ is of finite index, invoking \cite[5.2(iii)]{GG}, we infer that $\textrm{spli}(RG)=\textrm{spli}(RH)$. Since we have ${\widetilde{\textrm{Gcd}}_{R}H}\leq {\widetilde{\textrm{Gcd}}_{R}G}$ by Proposition \ref{prop38}, it suffices to show that ${\widetilde{\textrm{Gcd}}_{R}G}\leq {\widetilde{\textrm{Gcd}}_{R}H}$. For that, it suffices to assume that ${\widetilde{\textrm{Gcd}}_{R}H}$ is finite. Since $\textrm{spli}R<\infty$, invoking Theorem \ref{theorr}, we obtain that $\textrm{spli}(RH)<\infty$ and hence $\textrm{spli}(RG)<\infty$ as well. Using again Theorem \ref{theorr}, we infer that ${\widetilde{\textrm{Gcd}}_{R}G}<\infty$. Therefore, Proposition \ref{prop413} yields ${\widetilde{\textrm{Gcd}}_{R}G}=\underline{\textrm{cd}}_R G$ and ${\widetilde{\textrm{Gcd}}_{R}H}=\underline{\textrm{cd}}_R H$. Consequently, it suffices to show that $\underline{\textrm{cd}}_R G=\underline{\textrm{cd}}_R H$. Invoking \cite[Proposition 5]{In}, which holds over any commutative ring $R$, we conclude that  ${\widetilde{\textrm{Gcd}}_{R}G}={\widetilde{\textrm{Gcd}}_{R}H}$. The second inequality follows immediately from Proposition \ref{newprop}.
\end{proof}

\begin{Corollary}Let $G$ be a group and $H$ be a subgroup of $G$ of finite index. Then, ${\widetilde{\textrm{Gcd}}_{\mathbb{Z}}G}={\widetilde{\textrm{Gcd}}_{\mathbb{Z}}H}$ and $\textrm{Gcd}_{\mathbb{Z}}G=\textrm{Gcd}_{\mathbb{Z}}H$.
\end{Corollary}

\section{Hyperfinite extensions of PGF modules and the PGF dimension of an ascending union of groups}The goal of this section is to give an upper bound for the PGF dimension of a group, which is expressed as an ascending union of subgroups. For this reason, we study first the hyperfinite extensions of PGF modules and we prove that every hyper-${\tt PGF}(R)$ module is PGF. 
\subsection{Hyperfinite extensions of PGF modules}
As shown in \cite{SS}, the class  ${\tt {PGF}}(R)$ is closed under extensions. An inductive argument implies that an iterated extension of PGF modules is also PGF. In other words, for every nonnegative integer $n$ and every increasing filtration $0=M_0\subseteq M_1 \subseteq M_2 \subseteq \cdots \subseteq M_n=M$ of a module $M$ such that the quotient modules $M_{i+1}/M_i$ are PGF for every $i=0,1,\dots n-1$, we have $M\in {\tt PGF}(R)$. In this section we examine the case of an increasing filtration of infinite length.

Let $\mathfrak{C}$ be a class of modules. We say that a module $M$ is a hyper-$\mathfrak{C}$ module (or a hyperfinite extension of modules in $\mathfrak{C}$) if there exists an ordinal number $\alpha$ and an ascending filtration of $M$ by submodules $M_{\beta}$, which are indexed by the ordinals $\beta \leq \alpha$, such that $M_0 =0$, $M_{\alpha}=M$ and $M_{\beta}/M_{\beta -1}\in \mathfrak{C}$ (respectively, $M_{\beta}=\cup_{\gamma < \beta}M_{\gamma}$) if $\beta \leq \alpha$ is a successor (respectively, a limit) ordinal. In that case, we will say that $(M_{\beta})_{\beta \leq \alpha}$ is a continuous ascending chain of submodules with sections in $\mathfrak{C}$. If $\alpha=2$, we have the case of an extension of modules in $\mathfrak{C}$. Moreover, the direct sum of any family of modules in $\mathfrak{C}$ is a hyper-$\mathfrak{C}$ module and the class of hyper-$\mathfrak{C}$ modules is closed under extensions.

A class $\mathfrak{C}$ of modules is called $\Omega^{-1}$-closed, if for every $C\in \mathfrak{C}$ there exists a short exact sequence $$0\rightarrow C \rightarrow P \rightarrow D \rightarrow 0,$$ where $P$ is projective and $D\in \mathfrak{C}$. We note that the classes ${\tt PGF}(R)$ and ${\tt GProj}(R)$ are $\Omega^{-1}$-closed.
\begin{Proposition}\label{prop71}{\rm{(\cite[Proposition 2.1]{Emm2})}} Let $R$ be a ring and $\mathfrak{C}$ be an $\Omega^{-1}$-closed class consisting of Gorenstein projective modules. Consider a hyper-$\mathfrak{C}$ module $M$ which is endowed with a continuous ascending chain of submodules $(M_{\beta})_{\beta \leq \alpha}$ with sections in $\mathfrak{C}$ for some ordinal number $\alpha$. Then, there exists a hyper-$\mathfrak{C}$ module $N$ with a continuous ascending chain of submodules $(N_{\beta})_{\beta \leq \alpha}$ with sections in $\mathfrak{C}$ and a family of projective modules $(P_{\beta})_{\beta \leq \alpha}$, such that there exists a short exact sequence $$0\rightarrow M \rightarrow Q \rightarrow N \rightarrow 0$$ where $Q=\bigoplus_{\beta \leq \alpha}P_{\beta}$, and short exact sequences $$0\rightarrow M_{\beta +1}/ M_{\beta} \rightarrow P_{b+1} \rightarrow N_{\beta +1}/ N_{\beta} \rightarrow 0,$$ for every $\beta < \alpha$.
\end{Proposition}

\begin{Remark}\label{rem72} Since ${\tt PGF}(R)\subseteq {\tt GProj}(R)$ and the class ${\tt PGF}(R)$ is $\Omega^{-1}$-closed, Proposition \ref{prop71} implies that the class hyper-${\tt PGF}(R)$ is $\Omega^{-1}$-closed as well.
\end{Remark}

In the following lemma we give a simple criterion for a class of modules to be contained in ${\tt PGF}(R)$.

\begin{Lemma}\label{lem73}
Let $R$ be a ring and $\mathfrak{C}$ be an $\Omega^{-1}$-closed class. If $\textrm{Tor}^R_i(I,C)=0$, whenever $C\in\mathfrak{C}$, $I$ is injective and $i>0$, then $\mathfrak{C}$ consists of PGF modules.
\end{Lemma}

\begin{proof}Since the class $\mathfrak{C}$ is $\Omega^{-1}$-closed, for every $\mathfrak{C}$-module $C$ there exists an exact sequence $$\textbf{P}=0\rightarrow C \rightarrow P_{-1} \rightarrow P_{-2}\rightarrow \cdots \rightarrow P_{-n} \rightarrow \cdots ,$$ such that the module $P_{-n}$ is projective and the syzygies $K_{-n}=\textrm{Im}(P_{-n}\rightarrow P_{-n-1})$ are contained in $\mathfrak{C}$ for every $n\geq 1$. Let $I$ be an injective $R$-module. Then $\textrm{Tor}_1^R(I,K_{-n})=0$ for every $n\geq 1$ and hence the complex $I\otimes_R\textbf{P}$ is exact. Moreover, since $C\in\mathfrak{C}$ we have $\textrm{Tor}_i^R(I,C)=0$. Invoking \cite[Proposition 2.1]{St} we conclude that $C\in {\tt PGF}(R)$.\end{proof}

\begin{Corollary}\label{cor54}Let $R$ be a ring. Then, every hyper-${\tt PGF}(R)$ module is PGF.
\end{Corollary}

\begin{proof}Since the functors $\textrm{Tor}^R_i(I,\_\!\_)$ vanish on the class of PGF modules for all injective modules $I$ and all $i>0$ (see \cite[Proposition 2.1]{St}), using transfinite induction and the fact that the Tor functor preserves filtered colimits, we conclude that these functors vanish also on the class of hyper-${\tt PGF}(R)$ modules. Moreover, the class of hyper-${\tt PGF}(R)$ modules is $\Omega^{-1}$-closed (see Remark \ref{rem72}). Thus, the result is a consequence of Lemma \ref{lem73}.
\end{proof}

\subsection{The PGF dimension of an ascending union of groups}
Let $G$ be a group and $\mu$ be a limit ordinal such that there exists an ascending filtration of $G$ by subgroups $G_{\lambda}$ indexed by ordinals $\lambda \leq \mu$, where $G_{\mu}=G$ and $G_{\lambda}=\bigcup_{\kappa < \lambda} G_{\kappa}$ if $\lambda\leq \mu$ is a limit ordinal. We refer to the ascending chain of subgroups $(G_{\lambda})_{\lambda\leq \mu}$ as an exhaustive continuous ascending filtration of $G$.

\begin{Proposition}\label{prop55}Let $R$ be a commutative ring and consider a group $G$, a limit ordinal $\mu$ and an exhaustive continuous ascending filtration of $G$ by subgroups $(G_{\lambda})_{\lambda\leq \mu}$. If $M$ is an $RG$-module, which is PGF as an $RG_{\lambda}$-module for every $\lambda<\mu$, then $\textrm{PGF-dim}_{RG}M\leq 1$.
\end{Proposition}

\begin{proof}For every ordinal number $\lambda\leq \mu$ we regard $M$ as an $RG_{\lambda}$-module and we let $M_{\lambda}=\textrm{Ind}^{G}_{G_{\lambda}}M=RG\otimes_{RG_{\lambda}} M$. Since $G_{\mu}=G$ we have $M_{\mu}=M$. In the case where $\lambda<\mu$, the $RG_{\lambda}$-module $M$ is PGF and hence the $RG_{\lambda}$-module $M_{\lambda}$ is also PGF (see Lemma \ref{lem46}(i)). For every ordinal numbers $\kappa \leq \lambda \leq \mu$, the embedding $G_{\kappa}\hookrightarrow G_{\lambda}$ induces a surjective $RG$-linear map $M_{\kappa}\rightarrow M_{\lambda}$, and hence the family $(M_{\lambda})_{\lambda\leq \mu}$ is endowed with the structure of a direct system of $RG$-modules, with surjective structure maps. In particular, for every ordinal number $\lambda \leq \mu$ the embedding $G_0 \hookrightarrow G_{\lambda}$ induces a surjective $RG$-linear map $M_0 \rightarrow M_{\lambda}$. We let $N_{\lambda}=\textrm{Ker}(M_0\rightarrow M_{\lambda})$ and we observe that for every ordinal numbers $\kappa \leq \lambda \leq \mu$ we have the inclusion $N_{\kappa}\subseteq N_{\lambda}$ and $N_0=0$. We consider the following commutative diagram with exact rows:
	\[
	\begin{array}{ccccccccc}
	0 &{\longrightarrow} & N_{\kappa}& \longrightarrow & M_0 & \longrightarrow & M_{\kappa}& \longrightarrow & 0\\
	& & \downarrow & & \parallel & & \downarrow \\
	0 &{\longrightarrow} & N_{\lambda}& \longrightarrow&M_0 & \longrightarrow & M_{\lambda}& \rightarrow & 0
	\end{array}
	\]
	
	\smallskip \noindent Then, the snake lemma implies that $N_{\lambda}/N_{\kappa}\cong \textrm{Ker}(M_{\kappa}\rightarrow M_{\lambda})$ for every $\kappa \leq \lambda \leq \mu$. Since the class $\tt PGF$(RG) is closed under kernels of epimorphisms, we infer that the quotient modules $N_{\lambda}/N_{\kappa}$ are PGF for every $\kappa \leq \lambda \leq \mu$. Furthermore, if $\lambda \leq \mu$ is a limit ordinal, then $G_{\lambda}=\bigcup_{\kappa < \lambda} G_{\kappa}$ and hence $M_{\lambda}= {\lim\limits_{\longrightarrow}}_{\kappa < \lambda} M_{\kappa}$. Applying the colimit functor to the first row of the above commutative diagram, we obtain a commutative diagram with exact rows 
	\[
	\begin{array}{ccccccccc}
		0 &{\longrightarrow} &{\lim\limits_{\longrightarrow}}_{\kappa < \lambda} N_{\kappa}& \longrightarrow & M_0 & \rightarrow &{\lim\limits_{\longrightarrow}}_{\kappa < \lambda} M_{\kappa}& \longrightarrow & 0\\
		& & \downarrow & & \parallel & & \downarrow \\
		0 &{\longrightarrow} & N_{\lambda}& \longrightarrow&M_0 & \longrightarrow & M_{\lambda}& \longrightarrow & 0
	\end{array},
	\] and hence $N_{\lambda}= {\lim\limits_{\longrightarrow}}_{\kappa < \lambda} N_{\kappa}=\bigcup_{\kappa < \lambda} N_{\kappa}$. Consequently, the $RG$-module $N_{\mu}$ is a hyper-$\tt PGF$(RG) module. Invoking Corollary \ref{cor54}, we conclude that the $RG$-module $N_{\mu}$ is PGF as well. Thus, the short exact sequence $0\rightarrow N_{\mu}\rightarrow M_0\rightarrow M_{\mu}\rightarrow 0$ yields $\textrm{PGF-dim}_{RG}M\leq 1$.\end{proof}

\begin{Corollary}\label{cor56}Let $R$ be a commutative ring and consider a group $G$, a limit ordinal $\mu$ and an exhaustive continuous ascending filtration of $G$ by subgroups $(G_{\lambda})_{\lambda\leq \mu}$. Then, for every $RG$-module $M$ we have $\textrm{PGF-dim}_{RG}M\leq 1 +\textrm{sup}_{\lambda<\mu}\textrm{PGF-dim}_{RG_{\lambda}}M$.
\end{Corollary}

\begin{proof}
It suffices to assume that $\textrm{sup}_{\lambda<\mu}\textrm{PGF-dim}_{RG_{\lambda}}M=n$ is finite. Let $$\textbf{P}=\cdots \rightarrow P_n \rightarrow P_{n-1}\rightarrow \cdots \rightarrow P_0 \rightarrow M \rightarrow 0$$ be an $RG$-projective resolution of $M$ and $N=\textrm{Im}(P_n \rightarrow P_{n-1})$. Viewing $\textbf{P}$ as a projective resolution of the $RG_{\lambda}$-module $M$, we conclude that the $RG_{\lambda}$-module $N$ is PGF for every $\lambda<\mu$ (see \cite[Proposition 2.2]{DE}). Invoking Proposition \ref{prop55}, we infer that $\textrm{PGF-dim}_{RG}N\leq 1$, and hence $\textrm{PGF-dim}_{RG}M\leq \textrm{PGF-dim}_{RG}N + n \leq 1+n$ (see Lemma \ref{lem63}).\end{proof}

\begin{Corollary}\label{cor57}Let $R$ be a commutative ring and consider a group $G$ which is endowed with an exhaustive continuous ascending filtration by subgroups $(G_{\lambda})_{\lambda\leq \mu}$, for some limit ordinal $\mu$. Then, ${\widetilde{\textrm{Gcd}}_{R}G}\leq 1 +\textrm{sup}_{\lambda<\mu}{\widetilde{\textrm{Gcd}}_{R}G_{\lambda}}$.
\end{Corollary}

\begin{proof}This follows from Corollary \ref{cor56} by letting $M=R$.\end{proof}

\begin{Corollary}Let $R$ be a commutative ring and consider a group $G$ which is expressed as the union of an ascending sequence of subgroups $(G_n)_{n}$. Then, ${\widetilde{\textrm{Gcd}}_{R}G}\leq 1 +\textrm{sup}_{n}{\widetilde{\textrm{Gcd}}_{R}G_{n}}$.
\end{Corollary}

\begin{proof}This is a special case of Corollary \ref{cor57} for the limit ordinal $\mu=\omega$.\end{proof}

\section{Finiteness of PGF dimension and complete cohomology}

Mislin \cite{Mi} defined for any module $M$ complete cohomology functors
$\widehat{\mbox{Ext}}^{*}_{R}(M,\_\!\_)$ and a natural transformation
${\mbox{Ext}}^{*}_{R}(M,\_\!\_) \rightarrow
\widehat{\mbox{Ext}}^{*}_{R}(M,\_\!\_)$
as the projective completion of the ordinary Ext functors
${\mbox{Ext}}^{*}_{R}(M,\_\!\_)$. It follows that
${\widehat{\mbox{Ext}}}^{*}_{R}(M,N)=0$ if $\mbox{pd}_{R}N < \infty$. The equivalent approach given by Benson and Carlson \cite{BC} implies
that the elements in the kernel of the canonical map
${\mbox{Hom}}_{R}(M,N) \rightarrow \widehat{\mbox{Ext}}^{0}_{R}(M,N)$ are those
$R$-linear maps $f : M \rightarrow N$, which are such that the $R$-linear map
$\Omega^{n}f : \Omega^{n}M \rightarrow \Omega^{n}N$ induced by $f$ between the
$n$-th syzygy modules of $M$ and $N$ factors through a projective module for $n \gg 0$.

Let $M$ be a module such that $\textrm{PGF-dim}_R M<\infty$. Then, $\textrm{Gpd}_R M<\infty$ and hence there exists a totally acyclic complex of projective modules $\textbf{P}$, which coincides with a projective resolution $\textbf{Q}$ of $M$ in sufficiently large degrees. If
$\tau: \textbf{P} \rightarrow \textbf{Q}$ is a chain map which is the identity in
sufficiently large degrees, then the complete cohomology functors
$\widehat{\mbox{Ext}}^{*}_{R}(M,\_\!\_)$ may be computed as the cohomology groups
of the complex $\mbox{Hom}_{R}(\textbf{P},\_\!\_)$, while the canonical map
${\mbox{Ext}}^{*}_{R}(M,\_\!\_) \rightarrow
\widehat{\mbox{Ext}}^{*}_{R}(M,\_\!\_)$
is that induced by the chain map $\tau$ (see \cite[Lemma 2.4]{Mi}). In particular, the canonical map $\mbox{Ext}_{R}^{n}(M,N) \rightarrow \widehat{\mbox{Ext}}^{n}_{R}(M,N)$ is bijective if $n > \mbox{Gpd}_{R}M$ and surjective if $n = \mbox{Gpd}_{R}M$. If $ M$ is a PGF module, then $M$ is also Gorenstein projective and hence the canonical map $\mbox{Ext}_{R}^{n}(M,N) \rightarrow \widehat{\mbox{Ext}}^{n}_{R}(M,N)$ is bijective if $n > 0$ and surjective if $n = 0$.

The next result shows that the conditions on the finiteness of the projective
dimension of the module $K$ which appears in the short exact sequences
in \cite[Theorem 3.4(ii),(iii)]{DE} may be relaxed to the assertion that certain
elements of complete cohomology groups vanish.

\begin{Theorem}\label{Theor61}
	Let $M$ be an $R$-module. Then, the following conditions are equivalent.\begin{itemize}
		\item[(i)] $\textrm{PGF-dim}_{R}M < \infty$.
		\item[(ii)] There exists a short exact sequence of $R$-modules
		$0 \rightarrow K \stackrel{\iota}{\rightarrow} G
		\stackrel{p}{\rightarrow} M \rightarrow 0$,
		where $G$ is PGF, such that the image of the
		classifying element $\xi \in \mbox{Ext}_{R}^1(M,K)$ vanishes
		in the group $\widehat{\mbox{Ext}}^{1}_{R}(M,K)$.
		\item[(iii)] There exists a short exact sequence of $R$-modules
		$0 \rightarrow M \stackrel{\jmath}{\rightarrow} \Lambda
		\stackrel{q}{\rightarrow} L \rightarrow 0$,
		where $L$ is PGF, such that the image of
		$\jmath \in \textrm{Hom}_{R}(M,\Lambda)$ vanishes in the group
		$\widehat{\textrm{Ext}}^{0}_{R}(M,\Lambda)$.
	\end{itemize}
\end{Theorem}

\begin{proof}
	$(i)\Rightarrow(ii):$ Since $\textrm{PGF-dim}_{R}M < \infty$, \cite[Theorem 3.4]{DE} yields the existence of a short exact sequence of $R$-modules $0 \rightarrow K \xrightarrow{\iota} G\xrightarrow{p} M \rightarrow 0$, where $G$ is PGF and $\mbox{pd}_{R}K < \infty$. Thus, $\widehat{\mbox{Ext}}^{1}_{R}(M,K)=0$ and hence the image of the classifying element $\xi$ vanishes therein.
	
	$(ii)\Rightarrow(i):$ We consider the commutative diagram with exact rows
	\[
	\begin{array}{ccccc}
		\mbox{Hom}_{R}(G,K) & \stackrel{\iota^{*}}{\longrightarrow} &
		\mbox{Hom}_{R}(K,K) & \stackrel{\partial}{\longrightarrow} &
		\mbox{Ext}_{R}^{1}(M,K) \\
		\downarrow & & \downarrow & & \downarrow \\
		\widehat{\mbox{Ext}}^{0}_{R}(G,K) & \stackrel{\iota^{*}}{\longrightarrow} &
		\widehat{\mbox{Ext}}^{0}_{R}(K,K) & \stackrel{\partial}{\longrightarrow} &
		\widehat{\mbox{Ext}}^{1}_{R}(M,K)
	\end{array}
	\]
	Since the module $G$ is PGF, and hence is Gorenstein projective, we obtain that the additive map
	$\mbox{Hom}_{R}(G,K) \rightarrow \widehat{\mbox{Ext}}^{0}_{R}(G,K)$ is
	surjective. Then, the vanishing of the image of $\xi = \partial(1_{K})$ in the group $\widehat{\mbox{Ext}}^{1}_{R}(M,K)$
	yields the existence of an $R$-linear map 
	$f \in \mbox{Hom}_{R}(G,K)$, such that 
	$[f \iota] = [1_{K}]$ in $\widehat{\mbox{Ext}}^{0}_{R}(K,K)$. Letting $g=1_{K} - f \iota$, we obtain that the induced
	endomorphism $\Omega^{n}g$ of $\Omega^{n}K$ factors through a projective module
	for $n \gg 0$. Furthermore, it is clear that the endomorphism
	$\Omega^{n}(f \iota) =
	\! \left( \Omega^{n}f \right) \! \left( \Omega^{n}\iota \right)$
	of $\Omega^{n}K$ factors through $\Omega^{n}G$ for every $n>0$. Since $G$ is
	PGF, using induction on $n$ and \cite[Proposition 2.1(iii)]{St} we obtain that the syzygy module $\Omega^{n}G$ is $PGF$ as well. Since every projective module is PGF, we infer that
	$1_{\Omega^n K}=\Omega^{n}1_{K} = \Omega^{n}(g + f \iota) = \Omega^{n}g + \Omega^{n}(f \iota)$
	factors through a PGF module for $n \gg 0$. We conclude that 
	$\Omega^{n}K$ is a direct summand of a PGF module and hence 
	it is also PGF for $n \gg 0$ (see \cite[Proposition 2.3]{DE}). It 
	follows that $\mbox{PGF-dim}_{R}K \leq n$ for $n \gg 0$ and the short exact sequence
	$0 \rightarrow K \stackrel{\iota}{\rightarrow} G
	\stackrel{p}{\rightarrow} M \rightarrow 0$ then implies that $\mbox{PGF-dim}_{R}M \leq n + 1$ (see \cite[Proposition 2.4]{DE}).
	
	$(i)\Rightarrow(iii):$ Since $\textrm{PGF-dim}_{R}M < \infty$, \cite[Theorem 3.4]{DE} yields the existence of a short exact sequence of $R$-modules
	$0 \rightarrow M \stackrel{\jmath}{\rightarrow} \Lambda
	\stackrel{q}{\rightarrow} L \rightarrow 0$,
	where $L$ is PGF and $\mbox{pd}_{R}\Lambda < \infty$ (see \cite[Theorem 3.4]{DE}). Then, $\widehat{\mbox{Ext}}^{0}_{R}(M,\Lambda)$ is the trivial group and hence the
	image of $\jmath$ vanishes therein.
	
	$(iii)\Rightarrow(i):$ We consider the commutative diagram with exact rows
	\[\begin{array}{ccccc}
		\mbox{Hom}_{R}(L,\Lambda) & \stackrel{q^{*}}{\longrightarrow} & \mbox{Hom}_{R}(\Lambda,\Lambda) &
		\stackrel{\jmath^{*}}{\longrightarrow} & \mbox{Hom}_{R}(M,\Lambda) \\
		\downarrow & & \downarrow & & \downarrow \\
		\widehat{\mbox{Ext}}^{0}_{R}(L,\Lambda) & \stackrel{q^{*}}{\longrightarrow} &
		\widehat{\mbox{Ext}}^{0}_{R}(\Lambda,\Lambda) & \stackrel{\jmath^{*}}{\longrightarrow} &
		\widehat{\mbox{Ext}}^{0}_{R}(M,\Lambda)
	\end{array}\]
	Since the module $L$ is PGF, and hence is Gorenstein projective, we obtain that the additive map $\mbox{Hom}_{R}(L,\Lambda) \rightarrow \widehat{\mbox{Ext}}^{0}_{R}(L,\Lambda)$ is
	surjective. Then, the vanishing of the image of $\jmath = \jmath^{*}(1_{\Lambda})$ in the group $\widehat{\mbox{Ext}}^{0}_{R}(M,\Lambda)$ yields the existence of an $R$-linear map
	$f \in \mbox{Hom}_{R}(L,\Lambda)$, such that
	$[fq] = [1_{\Lambda}]$ in $\widehat{\mbox{Ext}}^{0}_{R}(\Lambda,\Lambda)$. Letting $g=1_{\Lambda} - fq$, we obtain that the induced endomorphism
	$\Omega^{n}g$ of $\Omega^{n}\Lambda$ factors through a projective module for $n \gg 0$.
	Furthermore, it is clear that the endomorphism
	$\Omega^{n}(fq) = \! \left( \Omega^{n}f \right) \! \left( \Omega^{n}q \right)$ of
	$\Omega^{n}\Lambda$ factors through $\Omega^{n}L$ for all $n$.  Since $L$ is
	PGF, using induction on $n$ and \cite[Proposition 2.1(iii)]{St} we obtain that the syzygy module $\Omega^{n}L$ is $PGF$ as well. Since every projective module is PGF, we infer that
	$1_{\Omega^n \Lambda}=\Omega^{n}1_{\Lambda} = \Omega^{n}(g + fq) = \Omega^{n}g + \Omega^{n}(fq)$ factors through a PGF module for $n \gg 0$. We conclude that $\Omega^{n}\Lambda$ is a direct summand of a PGF module and hence it is also PGF for $n \gg 0$ (see \cite[Proposition 2.3]{DE}). It follows that
	$\mbox{PGF-dim}_{R}\Lambda \leq n$ for $n \gg 0$ and the short exact sequence
	$0 \rightarrow M \stackrel{\jmath}{\rightarrow} \Lambda
	\stackrel{q}{\rightarrow} L \rightarrow 0 $
	then implies that $\mbox{PGF-dim}_{R}M \leq n$ (see \cite[Proposition 2.4]{DE}).\end{proof}

\begin{Proposition}\label{prop62}{\rm(\cite[Proposition 6.2]{Em-Ta})} Let $R$ be a commutative ring and consider an $R$-module $\Lambda$ and a split monomorphism of $R$-modules $\iota:R\rightarrow \Lambda$. We also consider the tensor powers $\Lambda^{\otimes n}$, $n\geq 1$, and define the direct system of $R$-modules and homomorphisms $$R\rightarrow \Lambda \rightarrow \Lambda^{\otimes 2} \rightarrow \cdots \rightarrow \Lambda^{\otimes n} \rightarrow \cdots,$$ with structural maps $\iota\otimes 1_{\Lambda^{\otimes n}}:\Lambda^{\otimes n}=R\otimes_R \Lambda^{\otimes n}\rightarrow \Lambda \otimes_R \Lambda^{\otimes n}=\Lambda^{\otimes n+1}$, for every $n\geq 0$. Let $\Lambda'$ be the direct limit of this direct system. Then:
	\begin{itemize}
		\item[(i)]The canonical map $\iota ' :R\rightarrow \Lambda'$ is a split monomorphism.
		\item[(ii)]If the module $\Lambda$ is projective, then $\Lambda '$ is projective as well.
	\end{itemize}
	\end{Proposition}

In the next theorem we give a characterization of the finiteness of ${\widetilde{\textrm{Gcd}}_{R}G}$ using the complete cohomology of the group $G$. 

\begin{Theorem}\label{theo43}Let $G$ be a group and $R$ be a commutative ring such that $\textrm{spli}R<\infty$. The following are equivalent:
	\begin{itemize}
		\item[(i)] ${\widetilde{\textrm{Gcd}}_{R}G}<\infty$.
		\item[(ii)] There exists an $R$-split monomorphism of $RG$-modules
		$\iota: R \rightarrow \Lambda$, where $\Lambda$ is $R$-projective, such that the image of $\iota \in \textrm{Hom}_{RG}(R,\Lambda)=\textrm{H}^{\,0} (G,\Lambda)$ vanishes in the group $\widehat{\textrm{Ext}}^{0}_{RG}(R,\Lambda)={\widehat{\textrm{H}}}^{\,0}(G,\Lambda)$.
		\item[(iii)]There exists a characteristic module for $G$ over $R$.
		\item[(iv)] $\textrm{silp}(RG)=\textrm{spli}(RG)< \infty$.
	\end{itemize}
\end{Theorem}

\begin{proof}$(i)\Rightarrow (ii):$ Since ${\widetilde{\textrm{Gcd}}_{R}G}<\infty$, invoking Theorem \ref{Theor61} and its proof, we infer that there exists a short exact sequence of $RG$-modules
	$0 \rightarrow R \stackrel{\iota}{\rightarrow} \Lambda
	\stackrel{q}{\rightarrow} L \rightarrow 0$,
	where $L$ is PGF and $\textrm{pd}_{RG} \Lambda <\infty$, such that the image of
	$\iota \in \textrm{Hom}_{RG}(R,\Lambda)=\textrm{H}^{0} (G,\Lambda)$ vanishes in the group
	$\widehat{\textrm{Ext}}^{0}_{RG}(R,\Lambda)=\widehat{\textrm{H}}^{0} (G,\Lambda)$. By Lemma \ref{lem1}, we infer that $L$ is PGF as $R$-module, and hence we have $\textrm{Ext}^1_R(L,R)=0$. Thus, the exact sequence $0 \rightarrow R \stackrel{\iota}{\rightarrow} \Lambda
	\stackrel{q}{\rightarrow} L \rightarrow 0$ is $R$-split. Moreover the finiteness of $\textrm{pd}_{R}\Lambda$ yields the finiteness of $\textrm{pd}_{R}L$ and so $L$ is $R$-projective (see \cite[Corollary 3.7(i)]{DE}). Consequently, the $R$-module $\Lambda$ is projective as well.
	
	$(ii)\Rightarrow (iii):$ Let $\textbf{P}$ be a projective resolution of the $RG$-module $R$ and $R_i=\Omega^i R$ be the corresponding syzygy modules of $R$, $i\geq 0$. Then, $\textbf{P}\otimes_R \Lambda$ is a projective resolution of $\Lambda$ with corresponding syzygies the $RG$-modules $R_i \otimes_R \Lambda$, $i\geq 0$. Lifting the $RG$-linear map $\iota$ we obtain a chain map $1\otimes \iota: \textbf{P}\rightarrow \textbf{P}\otimes_R \Lambda$. Thus, the vanishing of the image of $\iota$ under the canonical map $\textrm{H}^{0} (G,\Lambda)\rightarrow \widehat{\textrm{H}}^{0} (G,\Lambda)$ yields the existence of a nonnegative integer $m$ such that the map $1\otimes \iota: R_m \rightarrow R_m \otimes_R \Lambda $ factors through a projective $RG$-module $P$. Then, for every $RG$-module $K$, the map $1\otimes \iota \otimes 1 : R_m \otimes_R K = R_m \otimes_R R\otimes_R K \rightarrow R_m \otimes_R \Lambda \otimes_R K$ factors through the $RG$-module $P\otimes_R K$. We consider the direct system of the $RG$-modules $(\Lambda^{\otimes n})_n$ with structural maps $\iota \otimes 1_{\Lambda^{\otimes n}}:\Lambda^{\otimes n}=R\otimes_R \Lambda^{\otimes n}\rightarrow \Lambda^{\otimes n+1}$ for every $n\geq 0$ and let $\Lambda '$ be its direct limit. Using Proposition \ref{prop62}, we infer that the $RG$-module $\Lambda'$ is $R$-projective and the canonical map $\iota':R\rightarrow \Lambda'$ is an $R$-split monomorphism. The definition of $\Lambda'$ yields the existence of an isomorphism of $RG$-modules $\tau : \Lambda \otimes_R \Lambda' \rightarrow \Lambda'$, such that the composition $\Lambda'=R\otimes_R \Lambda' \xrightarrow{\iota\otimes 1} \Lambda \otimes_R \Lambda ' \xrightarrow{\tau} \Lambda'$ is the identity map of $\Lambda'$. Since the map $1\otimes \iota \otimes 1: R_m\otimes_R \Lambda' =R_m \otimes_R R \otimes_R \Lambda' \rightarrow R_m \otimes_R \Lambda \otimes_R \Lambda'$ factors through the $RG$-module $P\otimes_R \Lambda'$, we infer that the identity map of $R_m\otimes_R \Lambda'$ factors through $P\otimes_R \Lambda'$ as well. Since the $RG$-module $\Lambda'$ is $R$-projective, we obtain that the $RG$-module $P\otimes_R \Lambda'$ is projective, and hence the $RG$-module $R_m\otimes_R \Lambda'$ is also projective. Thus, the exact sequence $$0\rightarrow R_m \otimes_R \Lambda' \rightarrow P_{m-1}\otimes_R \Lambda' \rightarrow \cdots \rightarrow P_0\otimes_R \Lambda'\rightarrow R\otimes_R \Lambda'=\Lambda'\rightarrow 0$$ is a projective resolution of the $RG$-module $\Lambda'$ of length $m$. Then, $\textrm{pd}_{RG}\Lambda'\leq m$ is finite and $\Lambda'$ is a characteristic module for $G$ over $R$.
	
	$(iii)\Rightarrow (i):$ This follows from Corollary \ref{cor2}. 
	
	$(i)\Leftrightarrow (iv):$ This follows from Theorem \ref{theorr}.\end{proof}

\begin{Remark}\rm We note that the conditions (i), (ii) and (iii) of Theorem \ref{theo43} are also equivalent over any commutative ring such that $\textrm{sfli}R<\infty$. \end{Remark}

\section{Groups of type FP$_{\infty}$ and extensions}

In this section we study the special case where the group $G$ is of type FP$_{\infty}$ and has finite PGF dimension. We will show that these groups have characteristic modules of type FP. Moreover, we provide an analogue of Fel'dman's theorem \cite{Fe} concerning PGF dimensions of groups.

\subsection{Groups of type FP$_{\infty}$ with finite ${\widetilde{\textrm{Gcd}}}$ and characteristic modules}Consider a commutative ring $R$ of finite Gorenstein weak global dimension and a group $G$ of type FP$_{\infty}$ which has finite PGF dimension. The goal of this subsection is to prove the existence of a characteristic module for $G$ over $R$ of type FP.
\begin{Lemma}\label{lem71} Let $R$ be a commutative ring, $G$ a group and assume that there exists an
	$RG$-module $\Lambda$, which admits an $R$-split $RG$-linear monomorphism $\iota : R \rightarrow \Lambda$. We also consider a complex of $RG$-modules
	$\textbf{M}$ and assume that the induced complex of diagonal $RG$-modules
	$\textbf{M}\otimes_R \Lambda$ is contractible. Then, for every injective $RG$-module $I$ the induced complex $I\otimes_{RG}\textbf{M}$ is acyclic.
\end{Lemma}

\begin{proof} Let $L=\textrm{Coker}\iota$ and consider the $R$-split short exact sequence
	of $RG$-modules $0\rightarrow R \xrightarrow{\iota} \Lambda \rightarrow L \rightarrow 0.$
	We also consider an injective $RG$-module $I$ and the induced short exact sequence of
	$RG$-modules $0\rightarrow I \rightarrow \Lambda\otimes_{R}I \rightarrow L\otimes_{R}I\rightarrow 0.$
	Since the $RG$-module $I$ is injective, the exact sequence above splits and hence $I$ is
	a direct summand of $\Lambda\otimes_{R}I$. The acyclicity of the complex $I\otimes_{RG}\textbf{M}$ will therefore follow if we show the acyclicity of the complex $ (\Lambda\otimes_{R}I)\otimes_{RG}\textbf{M} \cong I\otimes_{RG}(\textbf{M}\otimes_{R} \Lambda)$. But the complex $\textbf{M}\otimes_R\Lambda$ is contractible by hypothesis and hence $(\Lambda\otimes_{R}I)\otimes_{RG}\textbf{M}$	is contractible as well. We conclude that the complex $I\otimes_{RG}\textbf{M}$ is acyclic.
\end{proof}

\begin{Lemma}\label{prop72}
Let $R$ be a commutative ring, $G$ be a group and $K$ be a finitely generated $RG$-module. We assume that there exists an $R$-projective $RG$-module $\Lambda$, which admits an $R$-split
$RG$-linear monomorphism $\iota: R \rightarrow \Lambda$ and is such that the diagonal $RG$-module $K\otimes_R \Lambda$ is projective. Then:
\begin{itemize}
	%\item[(i)] There exists a short exact sequence of $RG$-modules $0 \rightarrow K \rightarrow T \rightarrow  N \rightarrow 0$, such
	%that $T$ is finitely generated free and the diagonal $RG$-module $N\otimes_R \Lambda$ is projective.
	\item[(i)] There exists an exact sequence of $RG$-modules $$0\rightarrow K \rightarrow T_{-1} \rightarrow T_{-2} \rightarrow \cdots ,$$ where $T_i$ is finitely generated free and the image $K_i = \textrm{Im} (T_i \rightarrow T_{i-1})$ is such that the
	diagonal $RG$-module $ K_i\otimes_R\Lambda $ is projective for every $i \leq -1$.
	\item[(ii)]Every projective resolution $\textbf{P}$ of the $RG$-module $K$
	$$\textbf{P}=\cdots \rightarrow  P_m \rightarrow  P_{m-1} \rightarrow  \cdots \rightarrow  P_0 \rightarrow K \rightarrow 0$$
	may be completed to an acyclic complex 
	$$\textbf{{T}} = \cdots \rightarrow  P_m \rightarrow P_{m-1} \rightarrow \cdots \rightarrow P_0 \rightarrow T_{-1} \rightarrow T_{-2} \rightarrow \cdots $$
	with $K = \textrm{Im} (P_0 \rightarrow T_{-1})$, such that the complex $I\otimes_{RG}\textbf{{T}}$ is acyclic for every injective $RG$-module $I$ and $T_i$ is a finitely generated free $RG$-module for every $i\leq -1$.
\end{itemize} 
\end{Lemma}

\begin{proof}For the proof of (i), see \cite[Proposition 2.2 (ii)]{ET2}.
	
(ii) Splicing any projective resolution $\textbf{P}$ of $K$ with the
exact sequence of (i), we obtain an acyclic complex of projective $RG$-modules
$$\textbf{T}= \cdots \rightarrow P_m \rightarrow P_{m-1}\rightarrow \cdots \rightarrow P_0 \rightarrow T_{-1} \rightarrow T_{-2}\rightarrow \cdots$$ where $K=\textrm{Im}(P_0 \rightarrow T_{-1})$, such that $T_i$ is a finitely generated free $RG$-module for every $i\leq -1$. Since $\Lambda$ is an $R$-projective $RG$-module, the induced exact sequence of diagonal $RG$-modules $\textbf{T}\otimes_R\Lambda $ is also acyclic. The syzygies of $\textbf{T}\otimes_R\Lambda $ are the diagonal $RG$-modules $ K_i\otimes_R \Lambda $, where $K_i=\textrm{Im}(T_i\rightarrow T_{i-1})$, $i\in \mathbb{Z}$. By (i), the diagonal $RG$-modules $ K_i\otimes_R \Lambda $ are projective for every $i\leq -1$. We obtain that the complex $\textbf{T}\otimes_R\Lambda $ is contractible. Invoking Lemma \ref{lem71}, we infer that the complex $I\otimes_{RG}\textbf{T} $ is acyclic for every injective $RG$-module $I$.
\end{proof}

\begin{Proposition}\label{theo73}Let $R$ be a commutative ring such that $\textrm{sfli}R$ is finite. We also
	consider a group $G$ with ${\widetilde{\textrm{Gcd}}_{R}G}<\infty$ and an $R$-projective $RG$-module $M$ of type $FP_n$, where $n\geq {\widetilde{\textrm{Gcd}}_{R}G}$. Then, there exists an acyclic complex $\textbf{T}$ of projective $RG$-modules, such that $\textbf{T}$ coincides with a projective resolution of $M$ for every $i\geq n$, the $RG$-modules $T_i$ are finitely generated for every $i\leq n-1$ and the complex $I\otimes_{RG}\textbf{T}$
	is acyclic for every injective $RG$-module $I$.
\end{Proposition}

\begin{proof}Since $M$ is of type $FP_n$, there exists a projective resolution $\textbf{P}$ of $M$, such that the $RG$-module $K_n=\textrm{Im}(P_n\rightarrow P_{n-1})$ is finitely generated. Since $\textrm{sfli}R<\infty$ and ${\widetilde{\textrm{Gcd}}_{R}G}\leq n <\infty$, there exists
	an $R$-projective $RG$-module $\Lambda$ with $\textrm{pd}_{RG}\Lambda={\widetilde{\textrm{Gcd}}_{R}G}\leq n $ , which admits an $R$-split $RG$-linear monomorphism $\iota: R \rightarrow \Lambda$ (see Corollary \ref{cor1}). Since $\Lambda$ is $R$-projective, we obtain that $\textbf{P}\otimes_{R}\Lambda$ is a projective resolution of the $RG$-module $M\otimes_R \Lambda$. Moreover, the $RG$-module $M$ is $R$-projective and hence $\textrm{pd}_{RG}(M\otimes_R \Lambda)\leq \textrm{pd}_{RG}\Lambda \leq n$. Therefore, the $n$-th syzygy module $K_n \otimes_R \Lambda = \textrm{Im} (P_n \otimes_R \Lambda \rightarrow P_{n-1} \otimes_R \Lambda)$ of $\textbf{P}\otimes_{R}\Lambda$, is $RG$-projective.
	Since $K_n$ is a finitely generated $RG$-module with projective resolution $$\cdots \rightarrow P_m \rightarrow P_{m-1} \rightarrow \cdots \rightarrow P_n \rightarrow K_n \rightarrow 0$$
	and $K_n \otimes_R \Lambda$ is $RG$-projective, applying Lemma \ref{prop72}(ii) we obtain that there exists an acyclic complex of projective $RG$-modules $$\textbf{T}=\cdots \rightarrow P_m\rightarrow P_{m-1}\rightarrow \cdots \rightarrow P_n \rightarrow T_{n-1}\rightarrow T_{n-2}\rightarrow \cdots,$$
	where $K_n = \textrm{Im} (P_n \rightarrow T_{n-1})$, such that the $RG$-modules $T_i$ are finitely generated for every $i\leq n-1$ and the complex $I\otimes_{RG} \textbf{T}$
	is acyclic for every injective $RG$-module $I$.\end{proof}

The following corollary is an immediate consequence of Proposition \ref{theo73} and its proof.

\begin{Corollary}Let $R$ be a commutative ring such that $\textrm{sfli}R$ is finite. We also
	consider a group $G$ with ${\widetilde{\textrm{Gcd}}_{R}G}=n<\infty$ and an $R$-projective $RG$-module $M$ of type $FP_{\infty}$. Then, for every projective resolution $\textbf{P}$ of $M$, which consists of finitely generated free $RG$-modules in each degree, there exists an acyclic complex $\textbf{T}$ of finitely generated free $RG$-modules, such that $\textbf{T}$ coincides with $\textbf{P}$ for every $i\geq n$, and the complex $I\otimes_{RG}\textbf{T}$ is acyclic for every injective $RG$-module $I$.
\end{Corollary}

\begin{proof}We consider a projective resolution $\textbf{P}$ of $M$, consisting of finitely generated free $RG$-modules in each degree. Then, the acyclic complex $\textbf{T}$ of Proposition \ref{theo73} coincides with $\textbf{P}$ for every $i\geq n$, and consists of finitely generated free $RG$-modules in each degree. Indeed, in degrees $\geq n$, this follows since $\textbf{P}$ consists of finitely generated free $RG$-modules, while in degrees $\leq n-1$ this follows from Proposition \ref{theo73}. Moreover, the complex $I\otimes_{RG} \textbf{T}$ is acyclic for every injective $RG$-module $I$.\end{proof}

\begin{Corollary}\label{cor75}
Let $R$ be a commutative ring such that $\textrm{sfli}R$ is finite. We also
consider a group $G$ of type $FP_{\infty}$ with ${\widetilde{\textrm{Gcd}}_{R}G}=n<\infty$. Then, for every projective resolution $\textbf{P}$ of $G$, which consists of finitely generated free $RG$-modules in each degree, there exists an acyclic complex $\textbf{T}$ of finitely generated free $RG$-modules, such that $\textbf{T}$ coincides with $\textbf{P}$ for every $i\geq n$, and the complex $I\otimes_{RG}\textbf{T} $ is acyclic for every injective $RG$-module $I$.
\end{Corollary}

We will now show that a group of type FP$_{\infty}$
with finite PGF dimension has a characteristic module of type FP.
%As a consequence, it will follow that the Gorenstein cohomological dimension of such
%groups (at least in the case where k = Z is the ring of integers) is detected by a suitable
%choice of a field of coefficients

\begin{Theorem}\label{theo76}Let $R$ be a commutative ring such that $\textrm{sfli}R <\infty$ and $G$ be a group of type FP$_{\infty}$ over $R$ with ${\widetilde{\textrm{Gcd}}_{R}G}<\infty$. Then, there exists a characteristic module for $G$ over $R$ of type FP.
\end{Theorem}

\begin{proof}Let $\textbf{P}$ be a projective resolution of $G$ over $R$, which consists of finitely generated free modules in each degree. We also let ${\widetilde{\textrm{Gcd}}_{R}G}=n<\infty$. Invoking Corollary \ref{cor75}, we infer that there exists an acyclic complex $\textbf{T}$ of finitely generated free $RG$-modules, such that $\textbf{T}$ coincides with $\textbf{P}$ for every $i\geq n$, and the complex $I\otimes_{RG}\textbf{T}$ is acyclic for every injective $RG$-module $I$. Since all the syzygies of $\textbf{T}$ are PGF $RG$-modules, it follows from \cite[Corollary 4.5]{SS} that the complexes $\textrm{Hom}_{RG}(\textbf{T},P_i)$ are acyclic for every $i=0,1,\dots,n-1$, and hence the identity maps $T_i \rightarrow P_i$, $i\geq n$, extend to a chain map $\tau: \textbf{T}\rightarrow \textbf{P}$
$$\begin{array}{ccccccccccccccc}
	\cdots & \rightarrow & T_{n+1} & \rightarrow & T_n & \rightarrow & T_{n-1} & \rightarrow & \cdots & \rightarrow & T_0 & \rightarrow & T_{-1} & \rightarrow & \cdots \\
	& &  \parallel& & \parallel & & \downarrow  & & & &  \downarrow & &  \downarrow & & \\
	\cdots & \rightarrow & P_{n+1} & \rightarrow & P_n & \rightarrow & P_{n-1} & \rightarrow & \cdots & \rightarrow & P_0 & \rightarrow & 0 & \rightarrow & \cdots 
\end{array}$$
We may assume that the linear maps $\tau_i: T_i \rightarrow P_i$ are surjective for every $i<n$. Indeed, for every $i=0,1,\dots ,n-1$, we consider the contractible complex $\textbf{X}_i$, which consists of $P_i$ in degrees $i,i-1$ and $0$’s elsewhere with differential in degree $i$ given by
the identity map of $P_i$. Let $f_i : X_i \rightarrow P$ be the unique chain map whose component in degree $i$ is the identity map of $P_i$. Then, the direct sum $\textbf{T}'=\textbf{T}\oplus \textbf{X}_{n-1}\oplus \cdots \oplus \textbf{X}_{1} \oplus \textbf{X}_{0}$ is an acyclic complex of projective $RG$-modules which remains acyclic after applying the functor $\_\!\_\otimes_{RG}I$ for every injective $RG$-module $I$. Moreover the chain map $\textbf{T}' \rightarrow \textbf{P}$, which
is induced by $\tau$ and the $f_i$’s, $i = 0, 1, \dots, n - 1$, is surjective in degrees $\leq n-1$ and coincides with $\tau$ in degrees $\geq n$.

We consider the PGF $RG$-modules $K=\textrm{Coker}(T_{n+1}\rightarrow T_n)$, $M=\textrm{Coker}(T_{1}\rightarrow T_0)$ and $L=\textrm{Coker}(T_{0}\rightarrow T_{-1})$. Then, the chain map $\tau$ induces a commutative diagram 
$$\begin{array}{ccccccccccccccc}
	0 & \rightarrow & K & \rightarrow & T_{n-1} & \rightarrow & \cdots & \rightarrow & T_0 & \rightarrow & M & \rightarrow & 0 \\
	& & \parallel & & \downarrow  & & & &  \downarrow & &  \downarrow & & \\
	0 & \rightarrow & K & \rightarrow & P_{n-1} & \rightarrow & \cdots & \rightarrow & P_0 & \rightarrow & R & \rightarrow & 0 
\end{array}$$
with exact rows whose vertical maps are surjective. Letting $N=\textrm{Ker}(M\rightarrow R)$ and $Q_i=\textrm{Ker}(T_i\rightarrow P_i)$ for every $i=0,1,\dots, n-1$, we obtain the exact sequence 
\begin{equation}\label{eqqq}
	0\rightarrow Q_{n-1}\rightarrow \cdots \rightarrow Q_0 \rightarrow N \rightarrow 0.
\end{equation} Moreover, the projectivity of $P_i$ implies that $Q_i$ is a direct summand of $T_i$, and hence $Q_i$ is a finitely generated projective module for every $i=0,1,\dots,n-1$. 

We consider now the following pushout diagram:
$$\begin{array}{ccccccccc}
& & 0 & & 0 & & & & \\
& & \downarrow & & \downarrow & & & & \\
& & N & = & N & & & & \\
& & \downarrow & & \downarrow & & & & \\
0 & \rightarrow & M & \rightarrow & T_{-1} & \rightarrow & L & \rightarrow & 0\\
& & \downarrow & & \downarrow & & \parallel & & \\
0 & \rightarrow & R & \rightarrow & \Lambda & \rightarrow & L & \rightarrow & 0\\
& & \downarrow & & \downarrow & & & & \\
& & 0 & & 0 & & & &
\end{array}$$
Since the $RG$-module $L$ is PGF and $\textrm{sfli}R<\infty$, invoking Lemma \ref{lem1}, we infer that $L$ is PGF as $R$-module. Thus, $\textrm{Ext}^1_{R}(L,R)=0$ and the short exact sequence $0\rightarrow R \rightarrow \Lambda \rightarrow L\rightarrow 0$ is $R$-split. Moreover $\Lambda$ is PGF as $R$-module (see \cite[Proposition 2.3]{DE}). Splicing the exact sequence (\ref{eqqq}) with the short exact sequence $0\rightarrow N \rightarrow T_{-1}\rightarrow \Lambda \rightarrow 0$, we obtain an $RG$-projective resolution of $\Lambda$ $$0\rightarrow Q_{n-1}\rightarrow \cdots \rightarrow Q_0 \rightarrow T_{-1} \rightarrow \Lambda \rightarrow 0,$$ where the projective $RG$-modules $T_{-1}$ and $Q_i$, $i=0,1,\dots ,n-1$, are finitely generated. Hence, the $RG$-module $\Lambda$ is of type FP and $\textrm{pd}_{RG}\Lambda\leq n$. Since $\textrm{pd}_{R}\Lambda\leq\textrm{pd}_{RG}\Lambda\leq n$ is finite and $\Lambda$ is PGF as $R$-module, invoking \cite[Corollary 3.7(i)]{DE}, we infer that $\Lambda$ is $R$-projective. We conclude that $\Lambda$ is a characteristic module for $G$ over $R$ of type FP, as needed.\end{proof}

%\begin{Theorem}\label{theo79}Let $R$ be a principal ideal domain and $G$ be a group of type FP$_{\infty}$ over $R$ with ${\widetilde{\textrm{Gcd}}_{R}G}<\infty$. Then, there exists an $R$-algebra $F$, such that $F$ is a field and ${\widetilde{\textrm{Gcd}}_{F}G}={\widetilde{\textrm{Gcd}}_{R}G}$.\end{Theorem}

%\begin{proof}Since $\textrm{sfli}R<\infty$, Theorem \ref{theo76} yields the existence of a characteristic module $\Lambda$ of $G$ over $R$ of type FP. Invoking Proposition \ref{prop77}, we infer that there exists an $R$-algebra $F$, such that $F$ is a field and $\textrm{pd}_{FG}(\Lambda\otimes_R F)=\textrm{pd}_{RG} \Lambda$. Since $\Lambda\otimes_R F$ is a characteristic module for $G$ over $F$ (see Remark \ref{rem78}), Corollary \ref{cor38} yields ${\widetilde{\textrm{Gcd}}_{F}G}=\textrm{pd}_{FG}(\Lambda\otimes_R F)=\textrm{pd}_{RG} \Lambda={\widetilde{\textrm{Gcd}}_{R}G}$.\end{proof}

%\begin{Corollary}\label{cor610}Let $G$ be a group of type FP$_{\infty}$ over $\mathbb{Z}$ with ${\widetilde{\textrm{Gcd}}_{\mathbb{Z}}G}<\infty$. Then, there exists a field $F$, such that ${\widetilde{\textrm{Gcd}}_{F}G}={\widetilde{\textrm{Gcd}}_{\mathbb{Z}}G}$.\end{Corollary}
\subsection{PGF dimension of certain group extensions}Let $R$ be a commutative ring of finite Gorenstein weak global dimension. In this subsection, we provide an analogue of Fel'dman's theorem \cite{Fe} concerning PGF dimensions of groups. We will make use of the following result of \cite{ET2}.

\begin{Proposition}\label{prop611}{\rm(\cite[Corollary 4.4]{ET2})} Let $R$ be a commutative ring and consider an extension of groups $1\rightarrow N \rightarrow G \rightarrow Q \rightarrow 1$, where $N$ is a group of type FP$_{\infty}$ over $R$, and a nonnegative integer $i$.\begin{itemize}
		\item[(i)]If the $R$-module ${H}^i (N,RN)$ is projective, then the $RQ$-module ${H}^i (N,RG)$ is projective.
		\item[(ii)]If the $R$-module ${H}^i (N,RN)$ contains a copy of $R$ as a direct summand, then $RQ$-module ${H}^i (N,RG)$ contains a copy of $RQ$ as a direct summand.
	\end{itemize}
\end{Proposition} 

\begin{Remark}\label{rem612}\rm Let $R$ be a commutative ring and $N$ be a group with ${\widetilde{\textrm{Gcd}}_{R}N}=n<\infty$. Using \cite[Proposition 3.6]{DE} we infer that ${H}^i (N,RN)=0$, for every $i>n$. If $N$ is, in addition, of type FP$_\infty$ over $R$, then ${H}^n (N,RN)\neq 0$. Indeed, since ${\widetilde{\textrm{Gcd}}_{R}N}=n<\infty$ \cite[Proposition 3.6]{DE} implies that there exists a projective $RN$-module $P$, such that ${H}^n (N,P)\neq 0$. Since $P$ is a direct summand of a direct sum of copies of $RN$ and the cohomology functor ${H}^n (N,\_\!\_)$, for the FP$_{\infty}$-group $N$, commutes with direct sums, we conclude that ${H}^n (N,RN)\neq 0$. \end{Remark}

\begin{Theorem}\label{theo613}Let $R$ be a commutative ring such that $\textrm{sfli}R<\infty$ and consider an extension of groups $1\rightarrow N \rightarrow G \rightarrow Q \rightarrow 1$, such that \begin{itemize}
		\item[(i)] $N$ is a group of type FP$_{\infty}$ over $R$,
		\item[(ii)] both ${\widetilde{\textrm{Gcd}}_{R}N}$ and ${\widetilde{\textrm{Gcd}}_{R}Q}$ are finite and
		\item[(iii)]if ${\widetilde{\textrm{Gcd}}_{R}N}=n$, then the $R$-modules ${H}^i (N,RN)$ are projective for every $i<n$ and ${H}^i (N,RN)$ contains a copy of $R$ as an $R$-module direct summand.\end{itemize}
	Then, we have ${\widetilde{\textrm{Gcd}}_{R}G}={\widetilde{\textrm{Gcd}}_{R}N}+{\widetilde{\textrm{Gcd}}_{R}Q}$.
\end{Theorem}

\begin{proof}We follow closely the argument of \cite[Theorem 4.5]{ET2}. Let ${\widetilde{\textrm{Gcd}}_{R}Q}=m$. Then, we have $m=\textrm{sup}\{i\in\mathbb{N}\,|\,H^i(Q,F)\neq 0,\,F \,RQ\textrm{-free}\}$ (see \cite[Proposition 3.6]{DE}). In particular, there exists a free $RQ$-module $F$ such that $H^m(Q,F)\neq 0$. Let $r$ be the rank of the free $RQ$-module $F$. Using Proposition \ref{prop56}, we have the inequality ${\widetilde{\textrm{Gcd}}_{R}G}\leq m+n$. Hence, in order to prove that ${\widetilde{\textrm{Gcd}}_{R}G}= m+n$, it suffices to find an appropriate free $RG$-module $F'$ such that $H^{m+n}(G,F')\neq 0$ (see \cite[Proposition 3.6]{DE}). Let $F'$ be the free $RG$-module of rank $r$. We may compute the cohomology groups of $G$ with coefficients in $F'$ using the Lyndon-Hochschild-Serre spectral sequence $E^2_{pq}=H^p(Q,H^q(N,F'))\Longrightarrow H^{p+q}(G,F')$. The hypothesis that $N$ is of type FP$_{\infty}$, yields an isomorphism of $RQ$-modules $H^q (N,F')\cong H^q(N, RG)^{(r)}$, for every $q\geq 0$. Since ${\widetilde{\textrm{Gcd}}_{R}N}=n$ and $RG$ is $RN$-projective, we obtain that $H^q (N,F')=H^q(N, RG)^{(r)}=0$ for every $q>n$ (see \cite[Proposition 3.6]{DE}), and hence $E^2_{pq}=H^p(Q,H^q(N,F'))=H^p(Q,0)=0$, for every $q>n$. Invoking assumption (iii) and Proposition \ref{prop611}(i), we infer that the cohomology groups $H^q (N,F')=H^q(N, RG)^{(r)}$ are projective $RQ$-modules. Since ${\widetilde{\textrm{Gcd}}_{R}Q}=m$, we obtain that $E^2_{pq}=H^p(Q,H^q(N,F'))=0$ if $p>m$ and $q<n$. Thus, the $E^2$ page of the spectral sequence is concentrated on the square $[0,m]\times [0,n]$ and the line $\{(p,n): p\geq 0\}$. We obtain that $H^{n+m}(G,F')=E^{\infty}_{mn}=E^2_{mn}=H^m(Q,H^n(N,F'))$. Moreover, assumption (iii) and Proposition \ref{prop611}(ii) imply that the $RQ$-module $H^n(N,F')=H^n(N,RG)^{(r)}$ contains a copy of the free $RQ$-module $F$ of rank $r$ as a direct summand. Hence, the abelian group $H^m(Q,H^n(N,F'))$ contains a copy of $H^m(Q,F)$ as a direct summand. Since $H^m(Q,F)\neq 0$, we conclude that $H^{n+m}(G,F')=H^m(Q,H^n(N,F'))\neq 0$. 
\end{proof}

\begin{Corollary}\label{cor614}Let $R$ be a commutative ring such that $\textrm{sfli}R<\infty$ and consider an extension of groups $1\rightarrow N \rightarrow G \rightarrow Q \rightarrow 1$, such that \begin{itemize}
		\item[(i)] $N$ is a group of type FP$_{\infty}$ over $R$,
		\item[(ii)] both ${\widetilde{\textrm{Gcd}}_{R}N}$ and ${\widetilde{\textrm{Gcd}}_{R}Q}$ are finite and
		\item[(iii)]if ${\widetilde{\textrm{Gcd}}_{R}N}=n$, then the $R$-modules ${H}^i (N,RN)$ are free for every $i\leq n$.\end{itemize}
	Then, we have ${\widetilde{\textrm{Gcd}}_{R}G}={\widetilde{\textrm{Gcd}}_{R}N}+{\widetilde{\textrm{Gcd}}_{R}Q}$.
\end{Corollary}

\begin{proof}Since $N$ is of type FP$_{\infty}$ and ${\widetilde{\textrm{Gcd}}_{R}N}=n$, we have ${H}^n (N,RN)\neq 0$ (see Remark \ref{rem612}). Hence, the result is an immediate consequence of Theorem \ref{theo613}.
\end{proof}

\section{The PGF dimension of $\textsc{\textbf{lh}}\mathfrak{F}$-groups}Our goal in this section is to determine the PGF dimension ${\widetilde{\textrm{Gcd}}_{R}G}$ of an $\textsc{\textbf{lh}}\mathfrak{F}$-group $G$ over a commutative ring of finite Gorenstein global dimension, in terms of the projective dimension of the $RG$-module $B(G,R)$.

\begin{Definition}Let $R$ be a commutative ring and $G$ be a group.
	
	$\textrm{k}(RG):=\textrm{sup}\{\textrm{pd}_{RG}M \, : \, M\in \textrm{Mod}(RG), \, \textrm{pd}_{RH}M<\infty \, \textrm{for every finite} \,\, H\leq G\}$.
	
	$\textrm{fin.dim}(RG):=\textrm{sup}\{\textrm{pd}_{RG}M \, : \, M\in \textrm{Mod}(RG), \, \textrm{pd}_{RG}M<\infty\}$.
\end{Definition}

\begin{Lemma}\label{prop224}Let $R$ be a commutative ring and consider an $\textsc{\textbf{lh}}\mathfrak{F}$-group $G$. Then, $\textrm{k}(RG)\leq\textrm{fin.dim}(RG)$.
\end{Lemma}

\begin{proof}The proof is identical to that of \cite[Lemma 3.11]{Bis2} and holds for any commutative ring.
\end{proof}

\begin{Lemma}\label{prop225}Let $R$ be a commutative ring such that $\textrm{spli}R<\infty$ and $G$ be a group. Then, $\textrm{spli}(RG)\leq \textrm{k}(RG)$.
\end{Lemma}

\begin{proof}It suffices to assume that $\textrm{k}(RG)=n<\infty$. Let $I$ be an injective $RG$-module and $H$ be a finite subgroup of $G$. Then, $\textrm{spli}(RH)=\textrm{spli}R <\infty$, by Corollary \ref{cor44}. Since $I$ is injective as $RH$-module, we obtain that $\textrm{pd}_{RH}I<\infty$. It follows that $\textrm{pd}_{RG}I\leq \textrm{k}(RG)=n $, for every injective $RG$-module $I$. We conclude that $\textrm{spli}(RG)\leq \textrm{k}(RG)$, as needed.
\end{proof}

\begin{Proposition}\label{cor226}Let $R$ be a commutative ring such that $\textrm{spli}R<\infty$ and $G$ be an $\textsc{\textbf{lh}}\mathfrak{F}$-group. Then, $\textrm{k}(RG)=\textrm{spli}(RG)=\textrm{silp}(RG)=\textrm{fin.dim}(RG)$.
\end{Proposition}

\begin{proof}Since $RG\cong {(RG)}^{\textrm{op}}$, invoking \cite[Corollary 5.4]{DE} and \cite[Lemma 3.9(b)]{Bis2}, we infer that $\textrm{fin.dim}(RG)\leq \textrm{silp}(RG)\leq \textrm{spli}(RG)$. Therefore, using Lemma \ref{prop224} and Lemma \ref{prop225}, we have $\textrm{k}(RG)\leq \textrm{fin.dim}(RG)\leq \textrm{silp}(RG)\leq \textrm{spli}(RG)\leq \textrm{k}(RG)$ and we obtain the equalities.
\end{proof}

\begin{Remark}\label{Rem711} \rm Since the $RG$-module $B(G,R)$ is $R$-free and admits an $R$-split $RG$-linear monomorphism $\iota: R \rightarrow B(G,R)$, we infer that $B(G,R)$ is a characteristic module for $G$ over $R$ if and only if $\textrm{pd}_{RG} B(G,R)<\infty$.
\end{Remark}

\begin{Theorem}\label{Theo712}Let $R$ be a commutative ring such that $\textrm{spli}R<\infty$ and consider an $\textsc{\textbf{lh}}\mathfrak{F}$-group $G$. Then:
	\begin{itemize}
		\item[(i)] $B(G,R)$ is a characteristic module for $G$ iff ${\widetilde{\textrm{Gcd}}_R G}<\infty$,
		\item[(ii)] ${\widetilde{\textrm{Gcd}}_R G}=\textrm{Gcd}_R G=\textrm{pd}_{RG} B(G,R)$.
	\end{itemize}
\end{Theorem}

\begin{proof}(i) If $B(G,R)$ is a characteristic module, then Theorem \ref{theorr} implies that ${\widetilde{\textrm{Gcd}}_R G}<\infty$. Conversely, we assume that ${\widetilde{\textrm{Gcd}}_R G}<\infty$. Then, Proposition \ref{cor226} yields $\textrm{k}(RG)=\textrm{spli}(RG)\leq {\widetilde{\textrm{Gcd}}_R G}+ \textrm{spli}R <\infty$ (see Proposition \ref{prop2}). Since $B(G,R)$ is free as $RH$-module for every finite subgroup $H$ of $G$, we infer that $\textrm{pd}_{RG}B(G,R)\leq \textrm{k}(RG)<\infty$. Therefore, $B(G,R)$ is a characteristic module for $G$ over $R$ (see Remark \ref{Rem711}).
	
	(ii) Using (i) and Remark \ref{Rem711}, we have ${\widetilde{\textrm{Gcd}}_R G}=\infty$ if and only if $\textrm{pd}_{RG}B(G,R)=\infty$. If $\widetilde{\textrm{Gcd}}_R G<\infty$, then (i) implies that $B(G,R)$ is a characteristic module for $G$ over $R$, and hence, invoking Corollary \ref{cor38} and Proposition \ref{newprop}, we conclude that ${\widetilde{\textrm{Gcd}}_R G}=\textrm{Gcd}_R G=\textrm{pd}_{RG} B(G,R)$.
\end{proof}

\begin{Remark}\rm \label{Rem78} Let $\Lambda$ be a characteristic module for a group $G$ over $\mathbb{Z}$. Then the $RG$-module $\Lambda \otimes_{\mathbb{Z}} R$ is a characteristic module for $G$ over $R$. Indeed, since the $\mathbb{Z}G$-module $\Lambda$ is $\mathbb{Z}$-projective, the $RG$-module $\Lambda \otimes_{\mathbb{Z}} R$ is $R$-projective. Let $\textbf{P}$ be a $\mathbb{Z}G$-projective resolution of $\Lambda$ of finite length. Since $\Lambda$ is $\mathbb{Z}$-projective, $\textbf{P}$ is $\mathbb{Z}$-split and hence $\textbf{P}\otimes_\mathbb{Z} R$ is an $RG$-projective resolution of $\Lambda \otimes_\mathbb{Z} R$. It follows that $\textrm{pd}_{RG}(\Lambda \otimes_\mathbb{Z} R)<\infty$. Moreover, every $\mathbb{Z}$-split $\mathbb{Z}G$-linear monomorphism $\iota: \mathbb{Z} \rightarrow \Lambda$ induces an $R$-split $RG$-linear monomorphism $\iota \otimes 1 : R \rightarrow \Lambda \otimes_\mathbb{Z} R$.
\end{Remark}

\begin{Corollary}\label{cor88}Let $R$ be a commutative ring such that $\textrm{spli}R<\infty$ and $G$ be an $\textsc{\textbf{lh}}\mathfrak{F}$-group of type FP$_{\infty}$. Then, ${\widetilde{\textrm{Gcd}}_R G}=\textrm{Gcd}_R G=\textrm{pd}_{RG} B(G,R)<\infty$.
\end{Corollary}

\begin{proof}Invoking \cite[Corollary B.2(2)]{KRR}, which is also valid for $\textsc{\textbf{lh}}\mathfrak{F}$-groups, we infer that $\textrm{pd}_{\mathbb{Z}G} B(G,\mathbb{Z})<\infty$ and hence $B(G,\mathbb{Z})$ is a characteristic module for $G$ over $\mathbb{Z}$ (see Remark \ref{Rem78}). Consequently, the $RG$-module $B(G,R)=B(G,\mathbb{Z})\otimes_{\mathbb{Z}}R$ is a characteristic module for $G$ over $R$ and has finite projective $RG$-dimension (see Remark \ref{Rem78}). Thus, Theorem \ref{Theo712} yields ${\widetilde{\textrm{Gcd}}_{R}G}=\textrm{Gcd}_R G=\textrm{pd}_{RG} B(G,R)<\infty$, as needed.
\end{proof}

\begin{Corollary}Let $R$ be a commutative ring such that $\textrm{spli}R<\infty$ and $G$ be an $\textsc{\textbf{lh}}\mathfrak{F}$-group of type FP$_{\infty}$. Then, $\textrm{k}(RG)=\textrm{spli}(RG)=\textrm{silp}(RG)=\textrm{fin.dim}(RG)<\infty$. In particular, if $M$ is an $RG$-module, then $\textrm{pd}_{RG}M<\infty$ if and only if $\textrm{pd}_{RH}M<\infty$ for every finite subgroup $H$ of $G$.
\end{Corollary}

\begin{proof}In view of Corollary \ref{cor226}, it suffices to prove that $\textrm{spli}(RG)<\infty$. Invoking Corollary \ref{cor88} and Remark \ref{Rem711}, we infer that $B(G,R)$ is a characteristic module for $G$ over $R$. Thus, Theorem \ref{theorr} yields the finiteness of $\textrm{spli}(RG)$, as needed.\end{proof}

\section{Groups over which every Gorenstein projective module is PGF}Let $R$ be a commutative ring and $G$ be a group. In this final section we prove that every cofibrant $RG$-module is PGF. Under the assumptions $\textrm{gl.dim}R<\infty$ and $G$ be an  $\textsc{\textbf{lh}}\mathfrak{F}$-group or of type $\Phi_R$, we obtain that every Gorenstein projective $RG$-module is also a Gorenstein flat $RG$-module.

\begin{Definition}Let $M$ be an $RG$-module. Then $M$ is cofibrant if the $RG$-module $M \otimes_R B(G, R)$ is projective.
\end{Definition}

\begin{Proposition}
	Let $R$ be a commutative ring and $G$ be a group. Then, every cofibrant $RG$-module $M$ is PGF.
\end{Proposition}

\begin{proof}
	We let $B=B(G,R)$, $\overline{B}=\overline{B}(G,R)$ and consider an $RG$-module $M$ such that the $RG$-module $M\otimes_R B$ is projective. We also let $V_i=\overline{B}^{\otimes i}\otimes_R B$ for every $i\geq 0$, where $\overline{B}^{\otimes 0}=R$. Since the short exact sequence of $RG$-modules $0\rightarrow R \rightarrow B \rightarrow \overline{B}\rightarrow 0$ is $R$-split, we obtain for every $i\geq 0$ a short exact sequence of $RG$-modules of the form $$0\rightarrow M\otimes_R\overline{B}^{\otimes i}\rightarrow M\otimes_R V_i \rightarrow M\otimes_R \overline{B}^{\otimes i+1}\rightarrow 0.$$ Then, the splicing of the above short exact sequences for every $i\geq 0$ yields an exact sequence of the form
	\begin{equation}\label{eeq1}
		0\rightarrow M \xrightarrow{\alpha} M\otimes_R V_0 \rightarrow M\otimes_R V_1 \rightarrow M\otimes_R V_2 \rightarrow \cdots.
	\end{equation} 
	Since the $RG$-module $M\otimes_R B$ is projective and $\overline{B}$ is $R$-projective, we obtain that the $RG$-module  $M\otimes_R V_i\cong (M\otimes_R B)\otimes_R \overline{B}^{\otimes i}$ is projective for every $i\geq 0$. We also consider an $RG$-projective resolution of $M$
	\begin{equation*}
		\textbf{Q}=\cdots \rightarrow Q_2 \rightarrow Q_1 \rightarrow Q_0 \xrightarrow{\beta} M \rightarrow 0.
	\end{equation*} 
	Splicing the resolution $\textbf{Q}$ with the exact sequence (\ref{eeq1}), we obtain an acyclic complex of projective $RG$-modules 
	\begin{equation*}
		\mathfrak{P}=\cdots \rightarrow Q_2 \rightarrow Q_1 \rightarrow Q_0 \xrightarrow{\alpha \beta} M\otimes_R V_0 \rightarrow M\otimes_R V_1 \rightarrow M\otimes_R V_2 \rightarrow \cdots
	\end{equation*} 
	which has syzygy the $RG$-module $M$. It suffices to prove that the complex $I\otimes_{RG}\mathfrak{P}$ is acyclic for every injective $RG$-module $I$. Let $I$ be an injective $RG$-module. Then, the $R$-split short exact sequence of $RG$-modules $0\rightarrow R \rightarrow B \rightarrow \overline{B}\rightarrow 0$ yields an induced exact sequence of $RG$-modules $0\rightarrow I\rightarrow B \otimes_{R}I\rightarrow \overline{B}\otimes_{R} I\rightarrow 0$ which is $RG$-split. Thus, it suffices to prove that the complex $(B\otimes_{R}I)\otimes_{RG}\mathfrak{P}$ is acyclic. Since $B$ is $R$-projective, we obtain that the acyclic complex $\textbf{Q}\otimes_R B$ is a projective resolution of the projective $RG$-module $M\otimes_{R}B$. Hence, every syzygy module of $\textbf{Q}\otimes_R B$ is also a projective $RG$-module. Moreover, the $RG$-module $(M\otimes_R B)\otimes_R \overline{B}^{\otimes i}\cong (M\otimes_R \overline{B}^{\otimes i})\otimes_R B$ is projective for every $i\geq 0$. Consequently, every syzygy module of the acyclic complex
	\begin{equation*}
		\mathfrak{P}\otimes_R B =\cdots\rightarrow  Q_1\otimes_R B\rightarrow  Q_0\otimes_R B \rightarrow  M\otimes_R V_0 \otimes_R B\rightarrow  M\otimes_R V_1 \otimes_R B\rightarrow \cdots
	\end{equation*} is a projective $RG$-module and hence $\mathfrak{P}\otimes_R B$ is contractible. We conclude that the complex $(B\otimes_{R}I)\otimes_{RG}\mathfrak{P}\cong I\otimes_{RG}(\mathfrak{P}\otimes_{R}B)$ is acyclic, as needed.\end{proof}

\begin{Proposition}{\rm(\cite[Corollary 5.5]{Bis})} Let $G$ be a group in $\textsc{\textbf{lh}}\mathfrak{F}$ or of type $\Phi_R$ over a commutative ring $R$ of finite global dimension. Then, the class of weak Gorenstein projective $RG$-modules, the class of Gorenstein projective $RG$-modules and the class of Benson's cofibrant $RG$-modules all coincide.
\end{Proposition}

\begin{Theorem}\label{theo94}Let $G$ be a group in $\textsc{\textbf{lh}}\mathfrak{F}$ or of type $\Phi_R$ over a commutative ring $R$ of finite global dimension. Then, the class of Gorenstein projective $RG$-modules coincides with the class of PGF $RG$-modules and hence every Gorenstein projective $RG$-module is Gorenstein flat.
\end{Theorem}

\begin{proof}Let ${\tt Cof}(RG)$ be the class of cofibrant $RG$-modules. Invoking Propositions 7.7 and 7.8 we have ${\tt GProj}(RG)={\tt Cof}(RG)\subseteq {\tt PGF}(RG)$. Moreover, ${\tt PGF}(RG)\subseteq {\tt GProj}(RG)$ by \cite[Theorem 4.4]{SS}. We conclude that ${\tt GProj}(RG)={\tt PGF}(RG)$.\end{proof}

\begin{Corollary}Let $G$ be a group in $\textsc{\textbf{lh}}\mathfrak{F}$ or of type $\Phi_R$ over a commutative ring $R$ of finite global dimension. Then, for every $RG$-module $M$ we have $\textrm{Gfd}_{RG}M\leq \textrm{PGF-dim}_{RG}M=\textrm{Gpd}_{RG}M$.
\end{Corollary}

\section*{Acknowledgments}Research supported by the Hellenic Foundation for Research and Innovation (H.F.R.I.) under the ``1st Call for H.F.R.I. Research Projects to support Faculty members and Researchers and the procurement of high-cost research equipment grant”, project number 4226. The author wishes to thank the anonymous referee for useful comments.

%\section*{Declarations}The author states that there are no competing interests to declare.

\vspace{0.05in}

{\small {\sc Department of Mathematics,
             National and Kapodistrian University of Athens,
             Athens 15784,
             Greece}}

{\em E-mail address:} {\tt dstergiop@math.uoa.gr}


\begin{thebibliography}{99}
\bibitem{AB}M. Auslander, M. Bridger, Stable module theory, Memoirs of the American Mathematical Society 94, American Mathematical Society, Providence, RI 1969.   
%\bibitem{AM}L. Avramov, A. Martsinkovsky, Absolute, relative, and Tate cohomology of modules of finite Gorenstein dimension, \textit{Proc. London Math. soc.} (3) \textbf{85} (2002), no. 2, 393--440.
\bibitem{BDT}A. Bahlekeh, F. Dempegioti, O. Talelli, Gorenstein dimension and proper actions, \textit{Bull. London Math. Soc.} \textbf{41} (2009), 859--871.   
\bibitem{BC} D.J. Benson, J.F. Carlson,  Products in negative cohomology, \textit{J. Pure Appl. Algebra} \textbf{82} (1992), no.2, 107--129. 
\bibitem{Bis}R. Biswas, Benson's cofibrants, Gorenstein projectives and a related conjecture, \textit{Proc. Edinburgh Math. Soc.} \textbf{64}(4) (2021), 779--799.
\bibitem{Bis2}R. Biswas, On some cohomological invariants for large families of infinite groups, \textit{New York J. Math.} \textbf{27} (2021), 818--839.
\bibitem{Br}K.S. Brown, Cohomology of groups, Graduate Texts in Mathematics \textbf{87}, Springer, Berlin-Heidelberg-New York, 1982. 
\bibitem{CET}L.W. Christensen, S. Estrada, P. Thompson, Gorenstein weak global dimension is symmetric, \textit{Math. Nachrichten} \textbf{294} (2021), 2121--2128.
\bibitem{Kr2}J. Cornick, P.H. Kropholler, On complete resolutions, \textit{Topology Appl.} \textbf{78} (1997), 235--250.
\bibitem{KRR}J. Cornick, P.H. Kropholler, Homological finiteness conditions for modules over group algebras, \textit{J. London Math. Soc.} \textbf{58} (1998), 49--62.
\bibitem{DE}G. Dalezios, I. Emmanouil, Homological dimension based on a class of Gorenstein flat modules, arXiv:2208.05692 (2022) (to appear in Comptes Rendus Mathématique).
%\bibitem{Emm1}I. Emmanouil, On certain cohomological invariants of groups, \textit{Adv. Math.} \textbf{225} (2010), 3446--3462.
%\bibitem{Emma}I. Emmanouil, A homological characterization of locally finite groups, \textit{J. Algebra} \textbf{352} (2012), 167--172.
\bibitem{Emm3}I. Emmanouil, On the finiteness of Gorenstein homological dimensions, \textit{J. Algebra} \textbf{372} (2012), 376--396.
\bibitem{Emm2}I. Emmanouil, Precovers and orthogonality in the stable module category, \textit{J. Algebra} \textbf{478} (2017), 174--194.
\bibitem{Em-Ta}I. Emmanouil, O. Talelli, Finiteness criteria in Gorenstein homological algebra, \textit{Trans. Amer. Math. Soc.} \textbf{366} (2014), 6329--6351.
\bibitem{ET}I. Emmanouil, O. Talelli, Gorenstein dimension and group cohomology with group ring coefficients, \textit{J. London Math. Soc.} \textbf{97} (2018), 306--324.
\bibitem{ET2}I. Emmanouil, O. Talelli, On the Gorenstein cohomological dimension of group extensions, \textit{J. Algebra} \textbf{605} (2022), 403--428.
\bibitem{EJ}E.E. Enochs, O.M.G. Jenda, Gorenstein injective and projective modules, \textit{Math. Z.} \textbf{220} (1995), 611--633. 
\bibitem{EJ2}E.E. Enochs, O.M.G. Jenda, B. Torrecillas, Gorenstein flat modules, \textit{J. Nanjing Univ. Math. Biquarterly} \textbf{10} (1993), 1--9.
%\bibitem{Fa}F.T. Farrel, An extension of Tate cohomology to a class of infinite groups, \textit{J. Pure Appl. Algebra} {\bf 10} (1978), no. 2, 153--161.
\bibitem{Fe}G.I. Fel'dman, On the homological dimension of group algebras of soluble groups, \textit{Izv. Akad. Nauk SSR. Ser. Mat. Tom} \textbf{35} (1971), no. 6, 1231--1244.
\bibitem{GG}T.V. Gedrich, K.W. Gruenberg, Complete cohomological functors on groups, \textit{Topology Appl.} \textbf{25} (1987), 203--223.
%\bibitem{GT}R. Göbel, J. Trlifaj, Approximations and Endomorphism Algebras of Modules, Berlin-New York: Walter de Gruyter (2006).
%\bibitem{Go}F. Goichot, Homologie de Tate-Vogel equivariante, \textit{J. Pure Appl. Algebra} {\bf 82} (1992), no. 1, 39--64.
\bibitem{Ho}H. Holm, Gorenstein homological algebra, Ph.D. thesis, University of Copenhagen, Institute for Mathematical Sciences, 2004, ISBN 87-7834-587-1.
\bibitem{H1}H. Holm, Gorenstein homological dimensions, \textit{J. Pure Appl. Algebra} {\bf 189} (2004), 167--193.
\bibitem{In}B.M. Ikenaga, Homological dimension and Farrell cohomology, \textit{J. Algebra} \textbf{87} (1984), 422--457. 
\bibitem{St}I. Kaperonis, D.-D. Stergiopoulou, Finiteness criteria for Gorenstein flat dimension and stability, \textit{Comm. Algebra} (2023), DOI:10.1080/00927872.2023.2247104.
\bibitem{Kr}P.H. Kropholler, On groups of type (FP)$_{\infty}$, \textit{J. Pure Appl. Algebra} {\bf 90} (1993), 55--67.
\bibitem{MS}N. Mazza, P. Symonds, The stable category and invertible modules for infinite groups, \textit{Adv. Math.} \textbf{358} (2019), 106853, 26 pp.
\bibitem{Mi}G. Mislin, Tate cohomology for arbitrary groups via satellites, \textit{Topology Appl.} \textbf{56} (1994), no. 3, 293–-300.              
%\bibitem{Prest}M. Prest, Purity, Spectra and Localisation, Encyclopedia of Mathematics and its Applications, vol. 121, Cambridge University Press (2009).
\bibitem{SS}J. Šaroch, J. Š\'tovíček, Singular compactness and definability for $\Sigma$-cotorsion and Gorenstein modules, \textit{Sel. Math. New Ser.}\textbf{26} (2020), Paper No. 23. 

\bibitem{Ta}O. Talelli, A characterization of cohomological dimension for a big class of groups, \textit{J. Algebra} \textbf{326} (2011), 238--244. 
\bibitem{TA}O. Talelli, On groups of type $\Phi$, \textit{Arch. Math.} \textbf{89} (2007), no. 1, 24--32.
\bibitem{Tal}O. Talelli, On characteristic modules, Geometric and Cohomological Group Theory, P.H. Kropholler et al. (Eds.), London Math. Soc. Lecture Note Ser. \textbf{444}, Cambridge Univ. Press (2017), 172--181.
\bibitem{WL}J. Wang, L. Liang, A characterization of Gorenstein projective modules, \textit{Comm. Algebra} {\bf 44} (2016), 1420--1432.
\end{thebibliography}
\end{document}